\documentclass[11pt]{article}
	
    \newcommand{\blind}{0}
	
    \makeatletter
    \renewcommand\section{\@startsection {section}{1}{\z@}%
                                       {-3.5ex \@plus -1ex \@minus -.2ex}%
                                       {2.3ex \@plus.2ex}%
                                       {\normalfont\fontfamily{phv}\fontsize{16}{19}\bfseries}}
    \renewcommand\subsection{\@startsection{subsection}{2}{\z@}%
                                         {-3.25ex\@plus -1ex \@minus -.2ex}%
                                         {1.5ex \@plus .2ex}%
                                         {\normalfont\fontfamily{phv}\fontsize{14}{17}\bfseries}}
    \renewcommand\subsubsection{\@startsection{subsubsection}{3}{\z@}%
                                        {-3.25ex\@plus -1ex \@minus -.2ex}%
                                         {1.5ex \@plus .2ex}%
                                         {\normalfont\normalsize\fontfamily{phv}\fontsize{14}{17}\selectfont}}
    \makeatother
	
    \usepackage{amsmath}
    \usepackage{mathtools}
    \usepackage{amssymb}
    \usepackage[shortlabels]{enumitem}
    \usepackage{graphicx}
    \usepackage{comment}
    \usepackage{hyperref}
    \usepackage{color}
    \usepackage{soul}
    \usepackage{subcaption}
    \usepackage{schemata}
    \usepackage[ruled,vlined]{algorithm2e}
    \makeatletter
    \newcommand{\algorithmstyle}[1]{\renewcommand{\algocf@style}{#1}}
    \makeatother
    \usepackage{adjustbox}
    \usepackage{float}
    \usepackage{amsthm}
    \usepackage{chngcntr}
    \usepackage{url}
    \usepackage{apptools}
    \usepackage{thmtools}
    \usepackage{thm-restate}
    \usepackage{cleveref}
    \usepackage{wrapfig}
    \usepackage{enumitem}
    \usepackage{bm}
    \usepackage[round]{natbib}

    \DeclareMathOperator*{\argmin}{argmin} 
    \DeclareMathOperator*{\argmax}{argmax} 
    \makeatletter
    \newcommand*\bigcdot{\mathpalette\bigcdot@{.7}}
    \newcommand*\bigcdot@[2]{\mathbin{\vcenter{\hbox{\scalebox{#2}{$\m@th#1\bullet$}}}}}
    \makeatother

    \newtheorem{corollary}{Corollary}

	\allowdisplaybreaks

\begin{document}
		
		\def\spacingset#1{\renewcommand{\baselinestretch}%
			{#1}\small\normalsize} \spacingset{1}
		
		\if0\blind
		{
			\title{{\bf Decision-Dependent Distributionally Robust Markov Decision Process Method in Dynamic Epidemic Control}}
			\author{Jun Song, William Yang and Chaoyue Zhao \\
			Department of Industrial and Systems Engineering, University of Washington }
			\date{}
			\maketitle
		} \fi
		
		\if1\blind
		{

			\title{{\bf Decision-Dependent Distributionally Robust Markov Decision Process Method in Dynamic Epidemic Control}}
			\author{Author information is purposely removed for double-blind review}
			
\bigskip
			\bigskip
			\bigskip
			\begin{center}
				{\LARGE {\bf Decision-Dependent Distributionally Robust Markov Decision Process Method in Dynamic Epidemic Control}}
			\end{center}
			\medskip
		} \fi
		\bigskip
		
	\begin{abstract}
In this paper, we present a Distributionally Robust Markov Decision Process (DRMDP) approach for addressing the dynamic epidemic control problem. The Susceptible-Exposed-Infectious-Recovered (SEIR) model is widely used to represent the stochastic spread of infectious diseases, such as COVID-19. While Markov Decision Processes (MDP) offers a mathematical framework for identifying optimal actions, such as vaccination and transmission-reducing intervention, to combat disease spreading according to the SEIR model. However, uncertainties in these scenarios demand a more robust approach that is less reliant on error-prone assumptions. The primary objective of our study is to introduce a new DRMDP framework that allows for an ambiguous distribution of transition dynamics. Specifically, we consider the worst-case distribution of these transition probabilities within a decision-dependent ambiguity set. To overcome the computational complexities associated with policy determination, we propose an efficient Real-Time Dynamic Programming (RTDP) algorithm that is capable of computing optimal policies based on the reformulated DRMDP model in an accurate, timely, and scalable manner. Comparative analysis against the classic MDP model demonstrates that the DRMDP achieves a lower proportion of infections and susceptibilities at a reduced cost. 
	\end{abstract}
			
	\noindent%
	{\it Keywords:} Distributionally robust Markov decision process; distributionally robust optimization; SEIR model; real-time dynamic programming; epidemic control.
	\spacingset{1.5} 

\section{Introduction}
Infectious disease is a major contributing factor to human morbidity and mortality and it has a devastating impact on both human welfare and the economy. The COVID-19 outbreak has caused millions of infections and deaths worldwide, and has led to a $3.2\%$ global economy recession in 2020 \citep{world_economy_outlook}.
The SARS outbreak in 2003 was another major epidemic that incurred a worldwide loss of about 50 billion dollars \citep{knobler_2004_sarscost}. 

In this light, various mathematical models have been studied in the past few decades to understand the epidemic progression dynamics and to develop cost-effective interventions to control the spread of diseases, from the 2003 SARS outbreaks \citep[e.g.,][]{yan2008_optimalcontrolsars, gaff_2009_optimalcontrolvariousmodels}, to recent COVID‐19 pandemic \citep[e.g.,][]{hou_2020_covid, peng_2020_covid, zhao_2020_covid, giulia_2020_covid, henrik_2020_covid, biao_2020_covid, lopez_2021_covid_seir, grimm_2021_covid_seir}. Although these approaches provide powerful insights into building strategies to reduce the impact of epidemics on a macroscopic level, they are not specifically structured to assist real-time public health decision-making through rapidly evolving epidemics.

To address this issue, Markov Decision Process (MDP) has been proposed as a method to develop dynamic control policies in the stochastic environment of infectious disease propagation \citep{yaesoubi_2011_epidemic_control, parvin_2012_epidemic_control, sabbadin_2013_animal_epidemic_control, viet_2018_animal_epidemic_control}. In MDP-based epidemic control models, one important component is to model the transition probabilities, which are to characterize disease spreading dynamics through a population. However, in practice, it is difficult to estimate the transition probabilities accurately. Inaccurate estimation of transition probabilities can significantly deteriorate a model's ability to recommend effective intervention strategies.  However, there are a limited amount of MDP-based approaches that can cope with uncertainties in transition probabilities in epidemic control. 

One general method that is capable of addressing uncertainty in MDP is robust MDP \citep{white_1992_rmdp, nilim_2004_rmdp, nilim_2005_rmdp, iyengar_2005_rmdp}, which considers the worst-case realization of the uncertain parameter, and it has been applied to tackling uncertainties in transition probabilities in epidemic control problems. For example,  \cite{bhardwaj_2020_rmdp_covid} formulate a discrete-time epidemic model as a robust MDP and solve it with parameter-wise robust reinforcement
learning. In  \cite{bhardwaj_2020_rmdp_covid}, some environmental parameters, which are used to model the transition probabilities, are considered as uncertain parameters, and the solutions are based on the worst-case scenario of environmental parameters. However, robust MDP is often criticized as its overly conservativeness, since the worst-case happens rarely.

In this paper, we utilize a distributionally robust optimization approach to address the MDP model for epidemic control. There are two main advantages for the proposed Distributionally Robust Markov Decision Process (DRMDP) model. First, the DRMDP model allows the probability distribution of transition probabilities to be ambiguous. That is, instead of fixing the probability distribution of transition probabilities, we construct an ambiguity set of the distribution,  to embrace the uncertainty of transition probabilities and consequently increase the model's robustness. Second, the DRMDP model minimizes the total expected health and economic loss under the worst-case probability distribution of transition probabilities. This allows the model to give a conservative solution that hedges against the uncertainty of transition probabilities but not overly conservative as robust MDP, since DRMDP considers the stochastic nature of the random parameters instead of totally ignoring it as robust MDP does.

Different shapes and sizes of the ambiguity set of the distribution will affect the computational complexity and the robustness of the final solution. Among the few existing DRMDP studies, there are two main approaches to ambiguity set construction. The two approaches are to construct the ambiguity set based on the moment information \citep{xu_2010_drmdp, yu_2016_drmdp_counterpart, yang_2017_drmdp_moment, chen_2019_drmdp_general}, or distribution information \citep{osogami_2012_drmdp_kl, yang_2017_drmdp_wass, chen_2019_drmdp_general}. In epidemic settings,  limited information, especially in the early stage of epidemic evolution, makes it hard to construct an ambiguity set with density information. However, we can efficiently obtain estimates for the moments of epidemic model parameters in an Susceptible-Exposed-Infectious-Recovered (SEIR) model, so we utilize moment information to construct the ambiguity set in this paper.

In traditional distributionally robust optimization (DRO) settings, the ambiguity set of distributions is assumed to be exogenous, i.e., the distributions of random parameters are independent of what decisions have been made. However, during an epidemic, public health decisions can directly affect the spread of infectious diseases, which can be reflected by the transition probabilities. Therefore, to model the uncertainty of transition probabilities in dynamic epidemic control, we need to consider a decision-dependent (i.e., endogenous) ambiguity set. 

Previous work using DRO under endogenous uncertainty mainly focuses on single-stage or two-stage settings \citep{zhang_2016_endogenuous, royset_2017_endogenuous, luo_2020_endogenuous, noyan_2020_endogenuous, basciftci_2020_endogenuous}. For example, \cite{basciftci_2020_endogenuous} reformulate a two-stage endogenous distributionally robust facility location problem as a mixed-integer programming (MIP) using a McCormick envelope and linear decision rules. DRO under endogenous uncertainty with a multistage setting has not been explored until recently \citep{yu_2020_endogenuous, nakao_2020_drpomdp}. When extending endogenous DRO to a multistage setting, it remains challenging to design an efficient multistage algorithm. \cite{nakao_2020_drpomdp} model a differential equation based epidemic system as partially observable DRMDP, where a first moment ambiguity set of transition-observation probabilities is employed. Though partially observable DRMDP is defined with a continuous belief, \cite{nakao_2020_drpomdp} is not directly applicable to the continuous state case. In our paper, we consider the DRMDP epidemic model with a general state space, and we adopt a modified Real-Time Dynamic Programming (RTDP) algorithm to efficiently solve for optimal policies based on the corresponding DRMDP problem. 

To summarize, we utilize DRMDP under endogenous uncertainty to address the dynamic epidemic control problem. We highlight our contributions as follows: 

\begin{enumerate}
    \item {We propose DRMDP formulations under a decision-dependent ambiguity set to model the epidemic control problem. The ambiguity set is robust to the setting where the true transition probabilities are unknown}. We then derive a mixed-integer programming (MIP) reformulation of the DRMDP. 
    \item We develop a modified Real-Time Dynamic Programming (RTDP) method to efficiently compute optimal policies based on our DRMDP with a general state space. Since it computes an optimal partial epidemic control policy only for the states that are reachable from the initial state, this algorithm is significantly more efficient compared to the traditional value iteration algorithm which solves for the entire state space. 
    \item The numerical experiments verify that, as compared to the classic MDP, our DRMDP algorithm finds a better policy (i.e., a policy with lower total discounted health and economic loss) under misspecified distributions of transition probabilities. Furthermore, our results show that the DRMDP is more effective than the classic MDP in controlling the number of infectives. 
\end{enumerate}

\section{Preliminaries and Problem Setup}
\label{problem_setup}
Markov Decision Processes (MDP) are commonly used to solve sequential decision-making problems. In general, a finite-horizon discounted Markov Decision Process (MDP) is defined as a tuple $<T, \lambda, \mathcal{S}, \mathcal{A}, \bm{p}, \bm{r}>$, where $T$ is the time horizon, $\lambda$ is the discount factor, $\mathcal{S}$ is the state space, $\mathcal{A}$ is the action space, $\bm{p}$ is the transition probability between states depending on the action taken, and $\bm{r}$ is the reward associated with the state and the action taken. In this section, we model our problem as a continuous-state MDP model. We present the different MDP components in the context of the dynamic epidemic control problem below: 

\begin{itemize}
\item {\bf State}: The Susceptible-Exposed-Infectious-Recovered (SEIR) model is a classic model to describe the influenza epidemic, but can be generalized model to any infectious disease. The state of disease spread at stage $t$ is defined as $s^t = (p_S(t), p_E(t), p_I(t))$, where $p_S(t)$ denotes the proportion of susceptible individuals in the population at stage $t$, $p_E(t)$ denotes the percentage of exposed individuals at stage $t$, $p_I(t)$ denotes the percentage of infectious individuals at stage $t$. The state space can be defined as $\mathcal{S} = \{(p_S, p_E, p_I) \in \mathbb{R}_+^3 \hspace{1mm} | \hspace{1mm}  p_S + p_E + p_I \le 1 \}$. We assume the population size $N$ stays constant throughout the epidemic, and that susceptible individuals immediately gain lifelong immunity from vaccination.

\item {\bf Action}: We consider two categories of actions for controlling the spread of influenza: vaccination and transmission-reducing intervention. At each stage $t$, the decision maker will decide the proportion of susceptibles to vaccinate. $y_V (t) \in \{0,\dots,L\}$ represents {the scale of susceptibles to vaccinate} at stage $t$, where $y_V(t) = i$ corresponds to vaccinating $\frac{i}{L} \times 100\%$ susceptibles at stage $t$.  The transmission-reducing interventions are interventions that can be employed or lifted during the pandemic to reduce transmission, such as social distancing, wearing face masks, quarantining, closing schools etc. At each stage $t$, $y_R (t) \in \{0,\dots,M\}$ represents the scale of transmission-reducing interventions based on its strength, i.e., $0$ represents no transmission-reducing action, and $M$ represents the strongest transmission-reducing action. The action at stage $t$ can be defined as $a^t = (y_V(t), y_R (t))$, and the action space can be defined as $\mathcal{A} = \{(y_V, y_R) \in \mathbb{N}^2 | y_V \le L, y_R \le M\}$.

\item {\bf Transition Probabilities}: The transition probability $p_{as} (s') = P(s'|s,a)$ represents the probability of transitioning from the state-action pair $(s,a) \in \mathcal{S} \times \mathcal{A}$ in stage $t$ to the next state $s' \in \mathcal{S}$ in stage $t+1$. Its value depends on the stochastic epidemiological processes and on control measures implemented by the policy maker.  We define the vector $\bm{p}_{as} = (p_{as}(s'), s' \in \mathcal{S})^T$ as a collection of transition probabilities from the state-action pair $(s,a) \in \mathcal{S} \times \mathcal{A}$ to each $s' \in \mathcal{S}$. In other words, the $s'$ component of $\bm{p}_{as}$ is simply $p_{as}(s')$. Note that because $S$ is a continuous state space, $\bm{p}_{as}$ is an infinite-dimensional vector.

We will derive formulas for the transition probabilities using the underlying structure of the SEIR model. Let $n_B(t)$ denote the number of susceptible individuals who become exposed during time $t$, $n_C(t)$ denote the number of exposed individuals who become infectious during time $t$ and $n_D(t)$ denote the number of infectious individuals who recover during time $t$. The discrete-time stochastic SEIR model in \cite{lekone_2006_seir} specifies the following relationships:
\begin{equation*}
\begin{split}
    Np_S(t+1) &= Np_S(t)(1-\frac{y_V(t)}{L}) - n_B(t), \\
    Np_E(t+1) &= Np_E(t) + n_B(t) - n_C(t), \\
    Np_I(t+1) &= Np_I(t) + n_C(t) - n_D(t),
\end{split}
\end{equation*}
with $n_B(t) \sim \text{Bin} (Np_S(t) \times (1-\frac{y_V(t)}{L}), \phi(t))$, $n_C(t) \sim \text{Bin} (Np_E(t), \rho_C)$, and $n_D(t) \sim \text{Bin} (Np_I(t), \rho_D)$, where $\text{Bin}(.,.)$ denotes the binomial distribution, $\phi(t) = 1 - \exp(-(1-\alpha(t))\mu p_I(t)\beta)$, $\rho_C = 1 - \exp(-l_C)$, and $\rho_D = 1 - \exp(-l_D)$. The parameter $\mu$ denotes the contact rate when no transmission reduction method is used, and the parameter $\alpha(t)$ denotes the fraction reduction in the contact rate from transmission-reduction intervention. We assume that $\alpha(t) = \alpha_0 y_R (t)/M$, where $\alpha_0$ represents the maximum possible fractional reduction in the contact rate, which means $\alpha(t)$ has a linear relationship with the scale of the transmission-reducing method. The parameter $\beta$ denotes the probability that a susceptible individual becomes infected upon contact with an infectious individual. Lastly, the parameters $l_C$, $l_D$ denote the mean incubation period and the mean infectious period respectively. 

The nominal transition probability from $s^t$ to $s'$ given $a^t$ can be expressed as:
\begin{equation}
\begin{split}
    & p_{a^t s^t}^0 (s') = P(s' = (p_S,p_E,p_I) | s^t = (p_S(t), p_E(t), p_I(t)), a^t = (y_V(t), y_R(t)))  \\
    & = P(n_B(t) = Np_S(t) \times (1-\frac{y_V(t)}{L}) - Np_S) \\
    & \times P(n_C(t) = Np_S(t) \times (1-\frac{y_V(t)}{L}) + Np_E(t) - Np_S - Np_E) \\
    & \times P(n_D(t) = Np_S(t) \times (1-\frac{y_V(t)}{L}) + Np_E(t) + Np_I(t) - Np_S - Np_E - Np_I).
\end{split}
\label{true_pas}    
\end{equation} 
\item {\bf Rewards}: To represent the economic and health impact in each state for each action, we use reward matrices $\bm{r} \in \mathbb{R}^{|\mathcal{S}| \times |\mathcal{A}|}$. In this section, we will derive an expression for the components in the nominal reward matrix. We define the reward at stage $t$ to consist of the following components: 
\begin{enumerate}
    \item $c_V(t):= Q \frac{y_V(t)}{L} Np_S(t)$ is cost of vaccinations at stage $t$, where $Q$ is the unit price of vaccine.
    \item $c_R(t):= k_R y_R(t)$ is cost of implementing transmission-reduction method at stage $t$, where $k_R$ is a positive multiplier.
    \item $c_I(t):= W \mathbb{E} [Np_I(t) + n_C(t) - n_D(t)]$ is the expected total health loss and treatment cost due to infections at stage $t$, where $W$ is the health loss plus the treatment cost associated with a single infection. This formula can be seen as the cost for the expected number of infected people in the next time period.
\end{enumerate}
The nominal reward can then be expressed as:
\begin{equation}
    r_{a^t s^t}^0 = -c_V(t) - c_R(t) - c_I(t). 
\label{true_reward}
\end{equation}
\end{itemize}
In this paper, we propose and solve a distributionally robust policy under endogenous transition probability uncertainty. That is, the distribution $\mu_{as}$ of the transition probability $\bm{p}_{as}$ is not precisely known but is assumed to belong to an ambiguity set $\mathcal{D}_{as} \subseteq \mathcal{P}(\Delta({\mathcal{S})})$, where $\Delta({\mathcal{S})}$ is the probability simplex of set $\mathcal{S}$, and $\mathcal{P}(\Delta({\mathcal{S})})$ represents the set of all probability distributions with support $\Delta({\mathcal{S})}$. The objective is to find a policy $\pi: \mathcal{S} \rightarrow \mathcal A$ that determines corresponding vaccination and transmission-reducing intervention actions for different proportions of susceptible, exposed, and infectious individuals for each stage $t$. In this paper, to enhance the robustness of the model, we aim to find the optimal policy to maximize reward under the worst-case distribution $\mu_{as} \in \mathcal{D}_{as}$. That is, we consider a dynamic adversarial game between the public health decision-maker and nature, where at each stage, the public health decision-maker selects a epidemic control action $a \in \mathcal{A}$ to maximize total expected future reward while nature selects the distribution $\mu_{as}$ of $\bm{p}_{as}$ to minimize total expected future reward given the decision maker's action $a$. Note here $\mathcal{D}_{as}$ is an endogenous (or, decision-dependent) ambiguity set, i.e., it is depending on the action $a$ at state $s$. Most of the literature assumes the ambiguity set is exogenous (decision independent), however, this is not a reasonable assumption in our setting because control actions taken in each stage will significantly affect the spread of infectious diseases, and therefore affect the transmission probabilities. For example, if we take the action to vaccinate as many susceptible individuals as possible, close school, and employ social distancing, this should decrease the probability of entering a state with a higher proportion of infective
individuals. 
 
We will now derive an expression for the expected total reward in the adversarial game setting. Let $h^t = (s^1, a^1, \mu_{a^1 s^1}, \dots, s^{t-1}, a^{t-1}, \mu_{a^{t-1} s^{t-1}}, s^t)$ be the history of states and actions until stage $t$ and $H^t$ denote the set of all histories until stage $t$. The set of all history-dependent control policies for the decision maker is denoted by $\Pi = \{\pi = (\pi^1, \dots, \pi^{T-1}) \hspace{1mm}|\hspace{1mm} \pi^t: H^t \xrightarrow{} \mathcal{A}, \hspace{2mm} \forall t \in \{1, \dots, {T-1}\}\}$. Let $\tilde{h}^t = (s^1, a^1, \mu_{a^1 s^1}, \dots, s^{t-1}, a^{t-1},$ $ \mu_{a^{t-1} s^{t-1}}, s^t, a^t)$ be the extended history until stage $t$, with action $a_t$, and $\tilde{H}^t$ denote the set of all extended histories until stage $t$. The set of nature's admissible policies are defined as $\Gamma = \{\gamma = (\gamma^1, \dots, \gamma^{T-1}) \hspace{1mm} |\hspace{1mm} \gamma^t: \tilde{H}^t \xrightarrow{} \mathcal{D}_{a^t s^t}, \hspace{2mm} \forall t \in \{1, \dots, T-1\} \}$. Given a strategy pair $(\pi, \gamma) \in (\Pi \times \Gamma)$, we define the expected total rewards as 
\begin{equation}
    R_s[\pi, \gamma] = \mathbb{E}_{\gamma} \Bigg[\sum_{t=1}^{T-1} \big(\lambda^{t-1} r_{a^t s^t}\big) + \lambda^{T-1} \bar{R}(s^T) \hspace{1mm} | \hspace{1mm} s^1 = s\Bigg], \label{reward}
\end{equation}
where $a^t = \pi^t (h^t)$, $\mathbb{E}_\gamma$ denotes the expectation with respect to the probability measure induced by nature's strategy $\gamma$, $s$ is the initial state and $\bar{R}$ is a bounded terminal reward function. To obtain the values of the instantaneous reward, $r_{a^ts^t}$, we use linear regression with the nominal reward $r_{a^ts^t}^0$ defined in $\eqref{true_reward}$. We describe this procedure in more detail in Section \ref{sec:formulate_as_mip}.

Our problem can be modeled as a zero-sum two-player dynamic game problem, where the public health decision maker's objective improves if and only if nature's objective gets worse. Thus, the desired epidemic control policy can be obtained by solving the following optimization problem:
\begin{equation}
    \max_{\pi \in \Pi} \min_{\gamma \in \Gamma} R_s [\pi, \gamma]. 
\label{drmdp}
\end{equation}

\section{Discretization}
\label{sec:discretization}

In this section, we describe the process of discretizing the state space and transition probabilities for our DRMDP model.

To discretize the state space, we consider the cube $H_{\mathcal{S}}$, which contains the continuous state space $\mathcal{S}$. The cube $H_{\mathcal{S}}$ is defined as $H_{\mathcal{S}} := \{(p_S, p_E, p_I) \in \mathbb{R}_+^3 \hspace{1mm}|\hspace{1mm} p_S \leq 1, p_E \leq 1, p_I \leq 1\}$. It is important to note that by construction, $H_{\mathcal{S}}$ may contain some combinations of $(\tilde{p}_S, \tilde{p}_E, \tilde{p}_I)$ that are not in the actual state space $\mathcal{S}$. However, defining $H_{\mathcal{S}}$ as a cube allows us to easily partition it into smaller cubes with equal volume. Each smaller cube has an edge length of $\frac{1}{Y}$, resulting in a total of $Y^3$ equal-volume cubes.

To further partition each small cube into simplexes, we adopt the Kuhn triangulation method \citep{moore_1992_kuhn}. Kuhn triangulation is commonly used in the discretization of MDPs due to its efficient computation of interpolation weights \citep{scott_1996_kuhn_mdp, munos_2002_kuhn_mdp}. By applying Kuhn triangulation, we divide each small cube into six equal-volume simplexes.

As a result of this discretization process, the state space after discretization is represented by $\tilde{\mathcal{S}} = \{(\tilde{p}_S, \tilde{p}_E, \tilde{p}_I) \hspace{1mm}|\hspace{1mm} \tilde{p}_S, \tilde{p}_E, \tilde{p}_I \in \{\frac{0}{Y}, \frac{1}{Y}, \dots, \frac{Y}{Y}\}\}$. For simplicity, we use the notation $\xi_1, \dots, \xi_{|\tilde{\mathcal{S}}|}$ to represent the discrete states in $\tilde{\mathcal{S}}$. It is important to note that any state $s \in \mathcal{S} \backslash \tilde{\mathcal{S}}$ is included in exactly one simplex, and every corner state $\xi \in \tilde{\mathcal{S}}$ is included in at least one simplex.

We denote the union of simplexes that include the corner state $\xi \in \tilde{\mathcal{S}}$ as $U(\xi)$, and the simplex that includes state $s \in \mathcal{S} \backslash \tilde{\mathcal{S}}$ as $B(s)$. Moreover, every state within a simplex $B$ can be expressed as a convex combination of the corner states of that simplex. The set of corner states for simplex $B$ is denoted as $C(B)$.

Figure \ref{fig:parta} shows the representation of the state space as a unit cube, and Figure \ref{fig:partb} shows the partitioning of the state space for the case when $Y=4$, which shows the unit cube split into $64$ equal volume cubes in this example. 

\begin{figure}[H]
\centering
  \begin{subfigure}{0.4\textwidth}
    \includegraphics[width=1.1\linewidth]{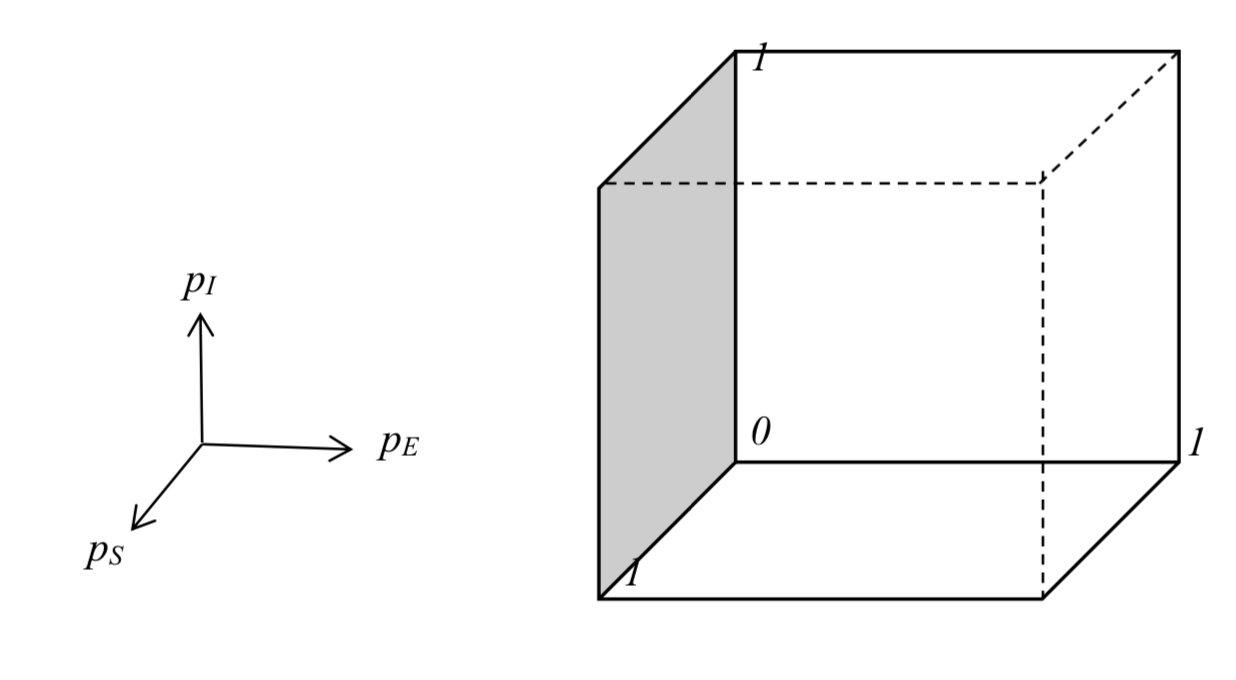}
    \caption{} \label{fig:parta}
  \end{subfigure}%
  \hspace{1cm}
  \begin{subfigure}{0.4\textwidth}
    \includegraphics[width=0.8\linewidth]{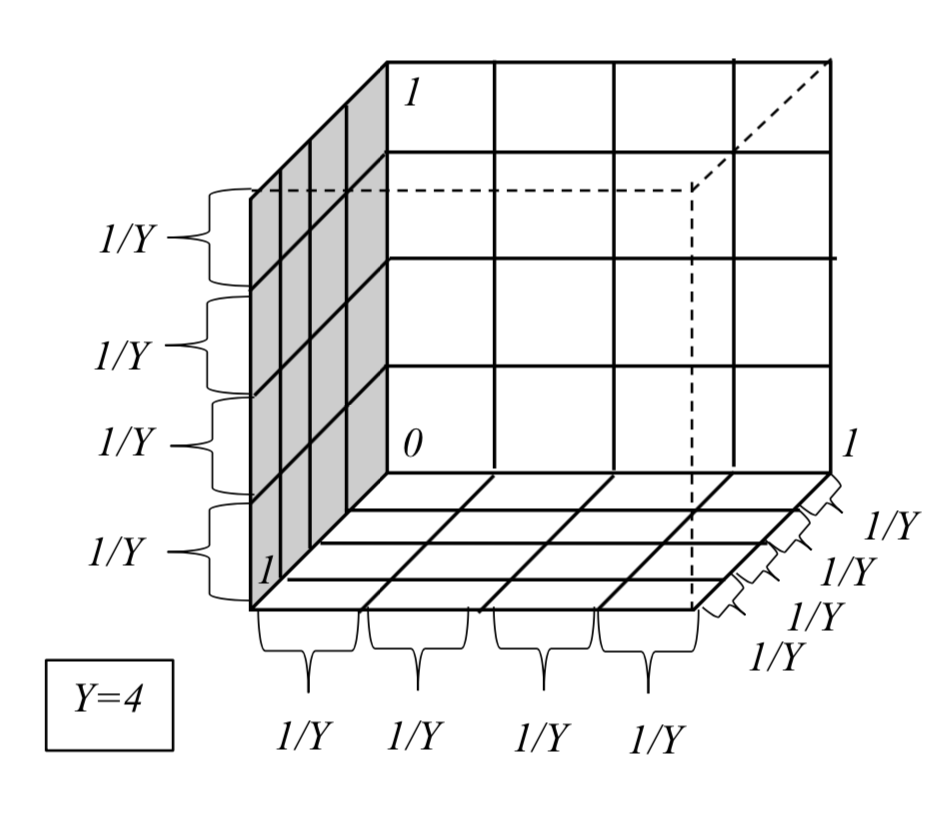}
    \caption{} \label{fig:partb}
  \end{subfigure}%
\caption{Partition of State Space where $Y=4$ }
\label{partition}
\end{figure}

Figure \ref{kt} represents the Kuhn triangulation of one of the smaller cubes shown in Figure \ref{fig:partb}. Figure \ref{fig:kta} shows the lines in which each simplex is divided by. And Figure \ref{fig:ktb} shows the separated $6$ simplexes. From the example in Figure \ref{fig:ktb}, we can see that $B(s) = III$, $U(\xi_0) = \{I,II,III,IV,V,VI\}$,         $U(\xi_1)= \{II,IV\}$,    $C(I) = \{\xi_0,\xi_4,\xi_5,\xi_7\}$, and  $C(II) = \{\xi_0,\xi_1,\xi_5,\xi_7\}.$
\begin{figure}[H]
\centering
  \begin{subfigure}{0.4\textwidth}
    \includegraphics[width=0.7\linewidth]{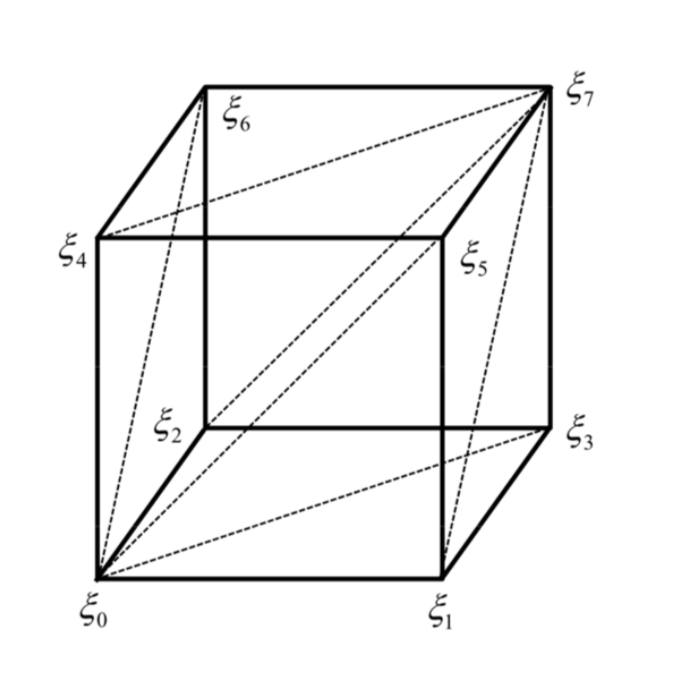}
    \caption{} \label{fig:kta}
  \end{subfigure}%
  \begin{subfigure}{0.4\textwidth}
    \includegraphics[width=0.7\linewidth]{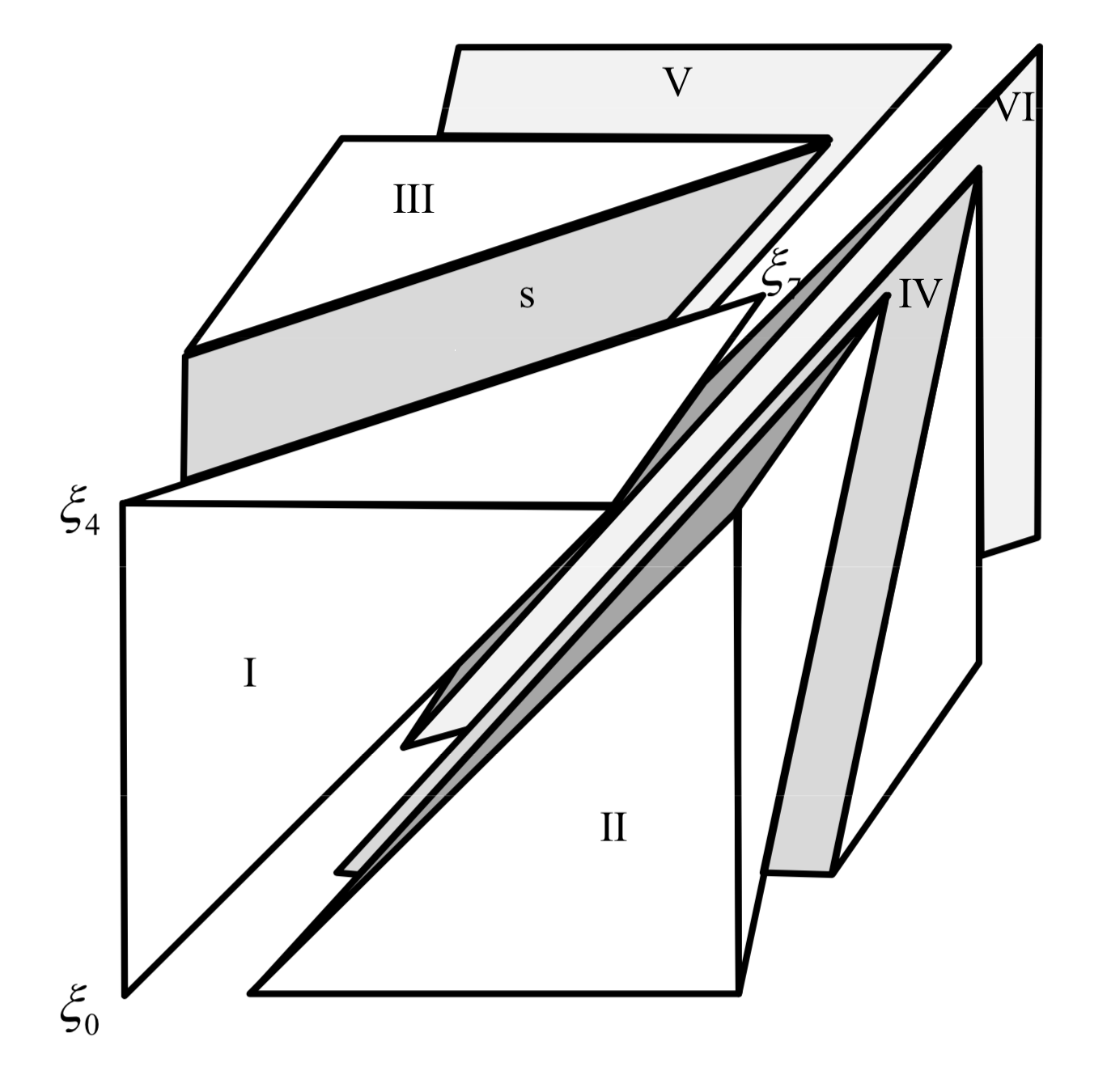}
    \caption{} \label{fig:ktb}
  \end{subfigure}%
\caption{Kuhn Triangulation \citep{nouri2011efficient}}
\label{kt}
\end{figure}

Therefore, the convex combination that represents $s \in B$ is 
$$s = \sum_{{\xi} \in C(B(s))} \theta_{s}^{\xi}  {\xi},\text{where} \sum_{{\xi} \in C(B(s))} \theta_{s}^{\xi} = 1. $$ $\theta_{s}^{\xi}$ is the weight of each $\xi$ in the convex combination. 
Based on the state space discretization, we define the discrete nominal reward and discrete terminal reward to be: $\tilde{r}_{a\xi_i}^0 = r_{a\xi_i}^0$ if $\xi_i \in \mathcal{S}$, and $\tilde{r}_{a\xi_i}^0 = 0$ otherwise; $\tilde{q}_R(\xi_i) = \bar{R}(\xi_i)$ if $\xi_i \in \mathcal{S}$, and $\tilde{q}_R(\xi_i) = 0$ otherwise, respectively. And we define the discrete nominal transition probability between $\xi_i, \xi_j \in \tilde{\mathcal{S}}$:  
\begin{itemize}
    \item If $\xi_i \in \mathcal{S}:$ 
\begin{equation}
    \tilde{p}^0_{a\xi_i} (\xi_j) :=  \int_{s \in U(\xi_j) \cap \mathcal{S}} \theta_{s}^{\xi_j} p^0_{a\xi_i} (s) ds.
\label{tilde_pas}
\end{equation}
    \item If $\xi_i \notin \mathcal{S}$: $\tilde{p}^0_{a\xi_i} (\xi_j) = 1$ if $\xi_i = \xi_j$; $\tilde{p}^0_{a\xi_i} (\xi_j) = 0$ if $\xi_i \ne \xi_j$.
\end{itemize}

We note that the discrete nominal transition probability is well defined since $\sum_{\xi \in \tilde{\mathcal{S}}} \tilde{p}^0_{a\xi_i} (\xi) = \sum_{\xi \in \tilde{\mathcal{S}}} \int_{s \in U(\xi) \cap \mathcal{S}} \theta_{s}^{\xi} p^0_{a\xi_i} (s) ds = 
\int_{s \in \mathcal{S}} \sum_{\xi \in C(B(s))} \theta_{s}^{\xi} p^0_{a\xi_i} (s) ds = \int_{s \in \mathcal{S}}p^0_{a\xi_i} (s) ds = 1$.

\section{Model Reformulation}
\label{sec:formulate_as_mip}

In this section, we consider DRMDP formulations with a decision-dependent ambiguity set, which can handle the case where only limited information of the epidemic statistics is available. To solve the decision-dependent uncertainty, we adopt the approach used in \cite{yu_2020_endogenuous, basciftci_2020_endogenuous} and reformulate the distributionally robust Bellman equation as a mixed integer programming (MIP) using McCormick or unary envelopes with linear decision rules. 

\subsection{Distributionally Robust Bellman Equation}
We first rewrite the expected reward-to-go function \eqref{drmdp} by replacing the value function $R_\xi [\pi, \gamma]$ by \eqref{reward}:
\begin{equation*}
V^t(\xi) = \max_{\pi \in \Pi} \min_{\gamma \in \Gamma} \mathbb{E}_{\gamma} \Bigg[\sum_{i=t}^{T-1} \lambda^{i-t} \tilde{r}_{a^i \xi^i} + \lambda^{T-t} \tilde{q}_{R}(\xi^T) \hspace{1mm} | \hspace{1mm} \xi^t = \xi \Bigg].
\end{equation*}
We assume that without loss of generality, maximizing the expected reward under worst admissible transition probability distribution is equivalent to solving the following distributionally robust Bellman equations: 
\begin{equation}
V^t(\xi) = \max_{a \in \mathcal{A}} \min_{\mu_{a\xi} \in \mathcal{D}_{a \xi}} \mathbb{E}_{\bm{p}_{a\xi} \sim \mu_{a\xi}} [\tilde{r}_{a\xi} + \lambda \bm{p}_{a\xi}^{\mbox{\tiny T}} \bm{V}^{t+1}],\label{drmdp_bellman}
\end{equation}

\begin{equation*}
Q^t (\xi,a) = \min_{\mu_{a\xi} \in \mathcal{D}_{a\xi}} \mathbb{E}_{\bm{p}_{a\xi} \sim \mu_{a\xi}} [\tilde{r}_{a\xi} + \lambda \bm{p}_{a\xi}^{\mbox{\tiny T}} \bm{V}^{t+1}], 
\end{equation*}

where $V^t(.)$ and $Q^t(., .)$ represent the distributionally robust state-value function and action-value function respectively, and $\bm{V}^{t+1} = (V^{t+1}(\xi'), \xi' \in \mathcal{\tilde{S}})^T$. 
The optimal action for the decision maker at stage $t$ then is: \begin{equation*}a^*=\argmax_{a \in \mathcal{A}} Q^t(\xi,a), \end{equation*} whereas the optimal distribution chosen by the nature is: \begin{equation*}\mu^*=\argmin_{\mu_{a\xi} \in \mathcal{D}_{a\xi}} \mathbb{E}_{\bm{p}_{a\xi} \sim \mu_{a\xi}} [\tilde{r}_{a\xi} + \lambda \bm{p}_{a\xi}^{\mbox{\tiny T}} \bm{V}^{t+1}]. \end{equation*}

\subsection{Ambiguity Set with First Moment Information}
\label{drmdp_dd1}
In practice, it is often the case that one has limited information about the transition probability distribution $\mu_{a\xi}$. In such situations, one can rely on the estimates of transition probability based on historical records or expertise domain knowledge. However, even if the mean value of transition probabilities can be estimated, the true distribution is still ambiguous. Therefore, we construct an ambiguity set to handle this uncertainty. We consider such an ambiguity set where the mean vector of the transition probabilities is restricted by decision-dependent bounds, and the true distribution of the transition probabilities can run adversely within the ambiguity set. The ambiguity set is constructed as follows:
\begin{subequations}
\begin{align}
\mathcal{D}_{a\xi} := \bigg\{ \mu_{a\xi} \in \mathcal{P}(\Delta(\mathcal{\tilde{S}})) \hspace{1mm}: \hspace{1mm}  & \bm{p}_{a\xi} \sim \mu_{a\xi} , \label{set_1_subeq1}\\
                 &  \tilde{\bm{\eta}}^L_{a\xi} \le \mathbb{E}[\bm{p}_{a\xi}] \le \tilde{\bm{\eta}}^U_{a\xi} \bigg\} \label{set_1_subeq2},
\end{align}
\label{set_1_mean_matching}
\end{subequations}
where $\bm{\tilde{\eta}}^L_{a\xi} $ and $\bm{\tilde{\eta}}^U_{a\xi}$ are the decision dependent lower and upper bounds, respectively.

\subsection{Reformulation of Bellman Equations}
In order to reformulate the Bellman equation \eqref{drmdp_bellman} into tractable formulations, under the setting of ambiguity set shown in \eqref{set_1_mean_matching}, we relax the hard constraint \eqref{set_1_subeq2} into a soft constraint and adjust the objective function by penalizing constraint violations. Then \eqref{drmdp_bellman} with the ambiguity set in \eqref{set_1_mean_matching} can be reformulated as:
\begin{subequations}
    \begin{align}
        \min_{\mu_{a\xi}, \bm{x} \ge 0} & \tilde{r}_{a\xi} + \int_{\Delta({\mathcal{\tilde{S}})}} \lambda \bm{p}_{a\xi}^{\mbox{\tiny T}} \bm{V}^{t+1} d \mu_{a\xi} (\bm{p}_{a\xi}) + k\bm{1}^{\mbox{\tiny T}} \bm{x}  \label{inner_prob_obj}\\
        s.t. \hspace{5mm} & \int_{\Delta({\mathcal{\tilde{S}})}} d \mu_{a\xi} (\bm{p}_{a\xi})  = 1, \label{inner_prob_prob}\\
        & \int_{\Delta({\mathcal{\tilde{S}})}} \bm{p}_{a\xi} d \mu_{a\xi} (\bm{p}_{a\xi}) - \tilde{\bm{\eta}}^U_{a\xi} \le \bm{x}, \label{inner_prob_u}\\
        & \tilde{\bm{\eta}}^L_{a\xi} - \int_{\Delta({\mathcal{\tilde{S}})}} \bm{p}_{a\xi} d \mu_{a\xi} (\bm{p}_{a\xi}) \le \bm{x},
        \label{inner_prob_l}
    \end{align}
    \label{inner_problem_set_1}
\end{subequations}
where the objective \eqref{inner_prob_obj} consists of the initial reward, the expected future reward over probability distribution $\mu_{a\xi}$ and a penalty term $k\bm{1}^{\mbox{\tiny T}} \bm{x}$. Here, $k$ represents a user-specified penalty coefficient to penalize the violation of constraint \eqref{set_1_subeq2}. \eqref{inner_prob_prob} represents the constraint \eqref{set_1_subeq1},  \eqref{inner_prob_u} and \eqref{inner_prob_l} are relaxation of \eqref{set_1_subeq2}. When $k\rightarrow \infty$, \eqref{inner_problem_set_1} will be equivalent to \eqref{drmdp_bellman} as $x$ will be $0$, i.e., no violation for constraint \eqref{set_1_subeq2}.

By utilizing the Lagrangian dualization approach, we reformulate the \eqref{inner_problem_set_1}, and we show the reformulation in the following theorem:

\begin{restatable}{thm}{thmambiguitysetone} If for any $a \in \mathcal{A}$, the ambiguity set defined in \eqref{set_1_mean_matching} is nonempty, then \eqref{inner_problem_set_1} can be reformulated as: 
\begin{equation}
    \begin{split}
        V^{t}(\xi) = \max_{a \in \mathcal{A}, \bm{w}, \bm{u}, q} & \tilde{r}_{a\xi} + q - \bm{w}^{\mbox{\tiny T}} \bm{\tilde{\eta}}^U_{a\xi} + \bm{u}^{\mbox{\tiny T}} \bm{\tilde{\eta}}^L_{a\xi} \\
        s.t. \hspace{5mm} 
        & q\bm{1}  \le \lambda \bm{V}^{t+1} + \bm{w} - \bm{u},  \\
        & \bm{w} + \bm{u} \le k\bm{1}, \\
        & \bm{w}, \bm{u} \ge \bm{0}.
    \end{split}
    \label{drmdp_bellman_set_1}
\end{equation}
\end{restatable}

Here, $\bm{1},\bm{V}^{t+1},\bm{w},\bm{u},\bm{\tilde{\eta}}_{a\xi}^U$ and $\bm{\tilde{\eta}}_{a\xi}^L$ are vectors with length $|\tilde{\mathcal{S}}|$. Therefore, the notation $\bm{w}^T\bm{\tilde{\eta}}_{a\xi}^U$ can be understood as

$$\bm{w}^T\bm{\tilde{\eta}}_{a\xi}^U:=\sum_{\xi' \in \mathcal{\tilde{S}}}w(\xi')\tilde{\eta}_{a\xi}^U(\xi'), $$
where $w(\xi')$ and $\tilde{\eta}_{a\xi}^U(\xi')$ represent the value of $\bm{w}^T$ and $\tilde{\bm{\eta}}_{a\xi}^U$ for state $\xi'$, respectively. $\bm{u}^T\bm{\tilde{\eta}}_{a\xi}^L$ is defined similarly. The proof of this theorem is provided in the online supplement.

To approximate \eqref{drmdp_bellman_set_1}, we adopt the linear decision rule \citep{holt_1960_ldr}, that is, we assume that $\bm{\tilde{\eta}}^U_{as}$, $\bm{\tilde{\eta}}^L_{as}$ and $\tilde{r}_{as}$ are linear functions of $a$:
\begin{align}
 &   \tilde{\eta}^U_{a\xi} (\xi') = \rho^\xi_{0}(\xi') + \sum_{i=1}^{N_a} \rho^\xi_{i}(\xi')a_i, \hspace{5mm} \forall{} \xi' \in \mathcal{\tilde{S}}, \label{affine_pas_u} \\
 &   \tilde{\eta}^L_{a\xi} (\xi') = \sigma^\xi_{0}(\xi') + \sum_{i=1}^{N_a} \sigma^\xi_{i}(\xi')a_i, \hspace{5mm} \forall{} \xi' \in \mathcal{\tilde{S}}, \label{affine_pas_l} \\
 &   \tilde{r}_{a\xi} =  \epsilon_0^\xi + \sum_{i=1}^{N_a} \epsilon_i^\xi a_i, \label{affine_r}
\end{align}
where $a_i$ represents the $i$-th dimension of action $a$, and $N_a$ is the total dimension of each $a$ in $\mathcal{A}$ ($N_a = 2$ in the case where we consider the two types of actions: vaccination and transmission-reduction intervention). To obtain the coefficients $\bm{\epsilon}^\xi$ in \eqref{affine_r}, we apply linear regression with training data $ \{\alpha_1, \dots, \alpha_n\}$ and target values $\{\tilde{r}_{\alpha_1 \xi}^0, \dots,  \tilde{r}_{\alpha_n \xi}^0 \}$, where $\tilde{r}_{a \xi}^0$ is the discrete nominal reward. Similarly, we apply linear regression with training data $\{\alpha_1, \dots, \alpha_n\}$ and target values
$\{\bm{\tilde{\eta}}^U_{{\alpha_1}\xi}, \dots, \bm{\tilde{\eta}}^U_{{\alpha_n}\xi}\}$ and $\{\bm{\tilde{\eta}}^L_{{\alpha_1}\xi}, \dots, \bm{\tilde{\eta}}^L_{{\alpha_n}\xi}\}$ to obtain coefficients $\bm{\rho}^\xi$ and $\bm{\sigma}^\xi$ in \eqref{affine_pas_u} and \eqref{affine_pas_l} respectively. Here we set the target values $\bm{\tilde{\eta}}^U_{{\alpha_i}\xi} = \bm{\tilde{p}}_{\alpha_i \xi}^0 + \bm{\delta}$ and  $\bm{\tilde{\eta}}^L_{{\alpha_i}\xi} = \bm{\tilde{p}}_{\alpha_i \xi}^0 - \bm{\delta}$ for $i = 1, \dots, n$, where $\bm{\delta} > 0$ is a pre-defined error bound and $\bm{\tilde{p}}_{a\xi}^0$ is the discrete nominal transition probability defined in \eqref{tilde_pas}. 

Thus, $\bm{w}^{\mbox{\tiny T}} \bm{\tilde{\eta}}^U_{a\xi}$,  and $\bm{u}^{\mbox{\tiny T}} \bm{\tilde{\eta}}^L_{a\xi}$ (which are terms of the objective of \eqref{drmdp_bellman_set_1}) can be expressed as: 
\begin{equation*}
\begin{split}
    & \bm{w}^{\mbox{\tiny T}} \bm{\tilde{\eta}}^U_{a\xi} = \sum_{\xi' \in \mathcal{\tilde{S}}}  \Big(\rho^\xi_{0}(\xi') + \sum_{i=1}^{N_a} \rho^\xi_{i}(\xi')a_i\Big) w(\xi'), \\
    & \bm{u}^{\mbox{\tiny T}} \bm{\tilde{\eta}}^L_{a\xi} = \sum_{\xi' \in \mathcal{\tilde{S}}}  \Big(\sigma^\xi_{0}(\xi') + \sum_{i=1}^{N_a} \sigma^\xi_{i}(\xi')a_i\Big) u(\xi').
\end{split}
\end{equation*}
To linearize the bilinear terms $a_i w(\xi')$ and $a_i u(\xi')$, we adopt two different methods: McCormick envelope relaxation and exact unary expansion. The corresponding mixed integer programming (MIP) formulations are presented in Corollary \ref{coro_mccormick} and Corollary \ref{coro_unary}, respectively, which we introduce below.

\noindent\textbf{McCormick Envelope Relaxation:}
  We replace the bilinear terms $a_i w(\xi')$ with $m^0_{i}(\xi')$ and $a_i u(\xi')$ with $m^1_{i}(\xi')$ by using McCormick envelopes \citep{mccormick_1976_envelope} $M^0_{i}(\xi')$ and $M^1_{i}(\xi')$ respectively for all $i = 1, \dots, N_a$, $\xi' \in \mathcal{\tilde{S}}$. We also utilize upper and lower bounds for decision variables $a_i,w(\xi'),$ and $u(\xi')$, which we denote as $\bar{a}_i,\bar{w}_{\xi'},\bar{u}_{\xi'}$ and $\underline{a}_i,\underline{w}_{\xi'},\underline{u}_{\xi'}$, respectively. In our setting, it is clear that $\bar{a}_1=L,\bar{a}_2=M,\underline{a}_1=\underline{a}_2=1,\bar{w}_{\xi'}=\bar{u}_{\xi'} = k,$ and $\underline{w}_{\xi'}=\underline{u}_{\xi'}=0 \hspace{2mm} \forall \xi' \in \mathcal{\tilde{S}}$, which leads us to the following corollary: 
  \begin{corollary}
\label{coro_mccormick}
 If for any $a \in \mathcal{A}$, the ambiguity set defined in \eqref{set_1_mean_matching} is nonempty, then the Bellman equation \eqref{drmdp_bellman} can be approximated as the following MIP formulation:
\begin{subequations} \label{drmdp_bellman_mip_set_1}
\begin{align}
            V^{t}(\xi) = \max_{a \in \mathcal{A}, \bm{w}, \bm{u}, q} & \tilde{r}_{a\xi} + q - \sum_{\xi' \in \mathcal{\tilde{S}}}  \bigg(\rho^\xi_{0}(\xi') w(\xi') + \sum_{i=1}^{N_a} \rho^\xi_{i}(\xi')m^0_{i}(\xi')\bigg) + \sum_{\xi' \in \mathcal{\tilde{S}}} \bigg(\sigma^\xi_{0}(\xi') u(\xi') \notag \\ 
            &+ \sum_{i=1}^{N_a} \sigma^\xi_{i}(\xi')m^1_{i}(\xi')\bigg)  \\
        s.t. \hspace{5mm} 
        & q\bm{1}  \le \lambda \bm{V}^{t+1} + \bm{w} - \bm{u},  \\
        & \bm{w} + \bm{u} \le k\bm{1}, \\
        & \bm{w}, \bm{u} \ge \bm{0}, \\
        & (m^0_{i}(\xi'), a_i, w(\xi')) \in M^0_{i}(\xi'), \hspace{5mm} \forall{} i \in [N_a], \xi' \in \mathcal{\tilde{S}}, \\
        & (m^1_{i}(\xi'), a_i, u(\xi')) \in M^1_{i}(\xi'), \hspace{5mm} \forall{} i \in [N_a], \xi' \in \mathcal{\tilde{S}},
\end{align}
\end{subequations}
where $[N_a]$ denotes the set $\{1, \dots, N_a\}$, and
\begin{equation*}
\begin{aligned}
    M^0_{i}(\xi') = \{&(m^0_{i}(\xi'), a_i, w(\xi')): \\ & m^0_{i}(\xi') \ge \underline{a}_i w(\xi') + a_i \underline{w}_{\xi'} - \underline{a}_i\underline{w}_{\xi'},  m^0_{i}(\xi') \ge \bar{a}_i w(\xi') + a_i \bar{w}_{\xi'} - \bar{a}_i\bar{w}_{\xi'}, \\
    & m^0_{i}(\xi') \le \bar{a}_i w(\xi') + a_i \underline{w}_{\xi'} - \bar{a}_i\underline{w}_{\xi'}, m^0_{i}(\xi') \le a_i \bar{w}_{\xi'} + \underline{a}_i w(\xi')- \underline{a}_i\bar{w}_{\xi'}. 
    \}, \\
    M^1_{i}(\xi') = \{&(m^1_{i}(\xi'), a_i, u(\xi')): \\ & m^1_{i}(\xi') \ge \underline{a}_i u(\xi') + a_i \underline{u}_{\xi'} - \underline{a}_i\underline{u}_{\xi'},  m^1_{i}(\xi') \ge \bar{a}_i u(\xi') + a_i \bar{u}_{\xi'} - \bar{a}_i\bar{u}_{\xi'}, \\
    & m^1_{i}(\xi') \le \bar{a}_i u(\xi') + a_i \underline{u}_{\xi'} - \bar{a}_i\underline{u}_{\xi'}, m^1_{i}(\xi') \le a_i \bar{u}_{\xi'} + \underline{a}_i u(\xi') - \underline{a}_i\bar{u}_{\xi'}. 
    \}. \\
\end{aligned}
\end{equation*}
\end{corollary}

\noindent \textbf{Exact Unary Expansion:} We note that the McCormick method used here provides additional relaxation of the original problem in \eqref{drmdp_bellman_set_1}. To have a more exact formulation, we utilize unary expansion \citep{gupte_2013_unary} in the following corollary: 
\begin{corollary}
\label{coro_unary}
 If for any $a \in \mathcal{A}$, the ambiguity set defined in \eqref{set_1_mean_matching} is nonempty, then the Bellman equation \eqref{drmdp_bellman} can be approximated as the following MIP formulation:
 \begin{subequations} \label{drmdp_bellman_mip_set_2}
\begin{align}
        V^{t}(\xi) = \max_{a \in \mathcal{A}, \bm{w}, \bm{u}, q} & \tilde{r}_{a\xi} + q - \sum_{\xi' \in \mathcal{\tilde{S}}}  \bigg(\rho^\xi_{0}(\xi') w(\xi') + \sum_{i=1}^{N_a} \rho^s_{i}(\xi')\sum_{j=1}^{A_i}\tau_j^i \hat{m}_j^{0i}(\xi')\bigg) \notag \\
        &\hspace{11mm}+ \sum_{\xi' \in \mathcal{\tilde{S}}} \bigg(\sigma^\xi_{0}(\xi') u(\xi') + \sum_{i=1}^{N_a} \sigma^\xi_{i}(\xi')\sum_{j=1}^{A_i}\tau_j^i \hat{m}_j^{1i}(\xi')\bigg)  \\
        s.t. \hspace{5mm} 
        & q\bm{1}  \le \lambda \bm{V}^{t+1} + \bm{w} - \bm{u},  \\
        & \bm{w} + \bm{u} \le k\bm{1}, \\
        & \bm{w}, \bm{u} \ge \bm{0}, \\
        &\sum_{j=1}^{A_i} \psi_j^{0i} = 1,\hspace{2mm}\sum_{j=1}^{A_i} \psi_j^{1i} = 1,\hspace{3mm}\forall i\in [N_a],\\
        &\psi_j^{0i},\psi_j^{1i} \in \{0,1\}, \hspace{3mm}\forall i\in [N_a], \hspace{1mm}j = 1,...,A_i\\
        & ({\hat{m}_j^{0i}(\xi')}, \psi_j^{0i}, w(\xi')) \in \hat{M}_j^{0i}(\xi'), \hspace{3mm} \forall{} i \in [N_a], \xi' \in \mathcal{\tilde{S}}, \\
        & ({\hat{m}_j^{1i}(\xi')}, \psi_j^{1i}, u(\xi')) \in \hat{M}_j^{1i}(\xi'), \hspace{3mm} \forall i \in [N_a], \xi' \in \mathcal{\tilde{S}},
\end{align}
\end{subequations}
where
\begin{equation*}
\begin{aligned}
    &\hat{M}_j^{0i}(\xi') = \{(\hat{m}_j^{0i}(\xi'), \psi_j^{0i}, w(\xi')): \hat{m}_j^{0i}(\xi') \geq 0, \hat{m}_j^{0i}(\xi') \geq w(\xi')+\bar{w}(\xi')(\psi_j^{0i}-1)\}, \\
    &\hat{M}_j^{1i}(\xi') = \{(\hat{m}_j^{1i}(\xi'),\psi_j^{1i},u(\xi')): \hat{m}_j^{1i}(\xi') \leq u(\xi'), \hat{m}_j^{1i}(\xi') \leq \bar{u}(\xi')\psi_j^{1i}\}, \\
    &\forall i\in [N_a], \hspace{1mm} j = 1,...,A_i,
\end{aligned}
\end{equation*}
where $A_i$ represents the number of actions for action type $i$, $\tau_j^i$ is the $j^{th}$ action for action type $i$ (in the epidemic control setting, $\tau_j^i$ is the $j^{th}$ item in $(0,1,2,...,A_i-1)$ for $i\in \{1,2\}$), and $\psi_j^{0i},\psi_j^{1i}$ are auxillary binary variables.
\end{corollary}

The main difference between \eqref{drmdp_bellman_mip_set_1} and \eqref{drmdp_bellman_mip_set_2} is that \eqref{drmdp_bellman_mip_set_2} introduces binary variables, $\psi_j^i$, to represent whether or not the discrete action is chosen. Therefore, we can represent the action chosen as the following:
$$a^i=\sum_{j=1}^{A_i} \tau_j^i\psi_j^i, \hspace{4mm}\text{where } \psi_j^i = \begin{cases} 1 \text{ if } \tau_j^i \text{ is chosen}\\0 \hspace{1mm} \text{otherwise} \end{cases}$$
Thus, we can represent the  bilinear terms in \eqref{drmdp_bellman_mip_set_2} as: {$$a^iw(\xi')=\sum_{j=1}^{A_i}  \tau_j^i\hat{m}_j^{0i}, \text{ and }a^ju(\xi')=\sum_{j=1}^{A_i}  \tau_j^i\hat{m}_j^{1i}$$}
for all $i \in [N_a]$ where
\begin{equation}\hat{m}_j^{0i} = \begin{cases}w(\xi')\text{ if action }j\text{ is chosen} \\ 0 \text{ otherwise}\end{cases}, \hat{m}_j^{1i} = \begin{cases}u(\xi')\text{ if action }j\text{ is chosen} \\ 0 \text{ otherwise}\end{cases}
\label{unary_logic}
\end{equation}

The logic of \eqref{unary_logic} is enforced by the unary envelopes $\hat{M}_j^{0i}$ and $\hat{M}_j^{1i}$, using binary indicator variables $\psi_j^{0i},\psi_j^{1i}$. This allows \eqref{drmdp_bellman_mip_set_2} to be an exact reformulation compared to \eqref{drmdp_bellman_mip_set_1}, which is a relaxation of \eqref{drmdp_bellman}. However, introducing binary variables adds computational complexity so there is an inherent trade-off between \eqref{drmdp_bellman_mip_set_1} and \eqref{drmdp_bellman_mip_set_2}. The realization of this trade-off is discussed further in Section \ref{sec:experiments}. 

\section{Real-Time Dynamic Programming}
\label{sec:rtdp}

To compute optimal policies based on the DRMDP model of the environment, we utilize the Real-Time Dynamic Programming (RTDP) algorithm \citep{barto_1995_rtdp}. RTDP is an online algorithm that combines heuristics search and dynamic programming to solve MDPs. Unlike traditional dynamic programming, RTDP does not require evaluating the entire state space to find an optimal solution. Instead, it computes a partial optimal policy for states reachable from the initial state, making it computationally efficient, especially for large state spaces.

It is important to note that the original RTDP algorithm iteratively solves the Bellman equation $\max_{a \in \mathcal{A}} \tilde{r}_{a\xi} + \lambda p_{a\xi}^\text{T} V^{t+1}$, whereas in our case, we consider the distributionally robust Bellman equation \eqref{drmdp_bellman}.

In the context of a finite-horizon MDP, visiting the same state at different stages can lead to different values of expected reward-to-go. To address this, we adopt the assumption proposed in \cite{xu_2010_drmdp} and treat multiple visits to a state as visiting different states. Therefore, in the following algorithm, we use the notation $(\xi,t)$ to denote the visit to state $\xi$ at stage $t$, and we represent the corresponding expected reward-to-go as $V(\xi,t) = V^t(\xi)$.

\begin{algorithm}[H]
\SetAlgoLined
Input: number of iterations $niter$, initial state $\xi_{init} \in \tilde{\mathcal{S}} \cap \mathcal{S}$ \\
Initialize $V((\xi,t)) = h((\xi,t))$ for all $\xi \in \tilde{\mathcal{S}}$, $t \in \{1,\dots,T\}$ \\
 \For{$i = 1, \dots, niter$}{
    Set $\xi = \xi_{init}$ \\
    \For{$t = 1, \dots, T-1$}{
        Solve the MIP relaxation in \eqref{drmdp_bellman_mip_set_1} or \eqref{drmdp_bellman_mip_set_2} to obtain optimal $V^*((\xi,t))$ and $a^*$ \\
        Update $V((\xi,t)) \xleftarrow{} V^*((\xi,t))$ \\
        Sample the next state $\xi' \in \tilde{\mathcal{S}} \cap \mathcal{S}$, and set $\xi = \xi'$
    }
 }
 \caption{RTDP with admissible heuristics $h((\xi,t))$}
 \label{rtdp_algorithm}
\end{algorithm}

In the provided algorithm, note that $\xi_{init}$ is defined outside the ``for" loops, indicating that for each iteration, we start from the same initial state. The function $h((\xi,t))$ represents a heuristic search pattern \citep{pearl_1984_heuristics, mausam_2012_heuristics}, which aims to find a solution specifically for the states reachable from the initial state by following an optimal policy. This approach is particularly effective for problems with large state spaces since it avoids the need to solve for the entire state space, as required by the traditional value iteration algorithm.

It is worth noting that if $h((\xi,t))$ is an \textit{admissible} heuristic function, meaning that its values are less than or equal to the true reward values, the algorithm will converge to the optimal solution. This result is formally stated below.

\begin{restatable}{thm}{rtdpconvergence} \citep{barto_1995_rtdp}
If the goal is reachable from each initial state and the heuristic function is admissible, the RTDP iterations will eventually yield optimal objective value and optimal controller's policy on the set of states reachable from the initial states. 
\label{rtdp_convergence}
\end{restatable} 

\section{Numerical Studies}
\label{sec:experiments}
 
In this section, we evaluate the effectiveness of our DRMDP model and RTDP algorithm in addressing the epidemic control problem.

We begin with Subsection {\ref{model_comparison}}, where we conduct a comprehensive performance comparison between our DRMDP model and traditional MDP-based models for addressing the epidemic control problem. Specifically, in Subsection {\ref{mccormick_vs_unary_model}}, we analyze the performance of two formulations of the DRMDP model: one utilizing the relaxed McCormick envelope (Corollary {\ref{coro_mccormick}}) and the other using the exact unary envelope (Corollary {\ref{coro_unary}}). This analysis helps us determine the optimal formulation for the DRMDP model before proceeding with the comparison. Moving on to Subsection {\ref{mdp_based_model_comparison}}, we present a comprehensive performance comparison between our DRMDP model and classic MDP as well as robust MDP models. We evaluate the models based on their effectiveness in controlling the number of infectives, the resulting epidemic control policies, and the total discounted health and economic loss. To assess the robustness of the models, we consider various scenarios, including accurate or inaccurate knowledge of transition probabilities.

Additionally, to assess the effectiveness and efficiency of RTDP when applied to the DRMDP model, we perform a comparative analysis in Subsection {\ref{algo_comparison}} between RTDP algorithm and the conventional dynamic programming (DP) approach using backward induction.

Lastly, in Subsection {\ref{sensitivity}}, we conduct a sensitivity analysis of the DRMDP model for the epidemic control problem. By systematically varying the values of various input parameters, we examine their influence on the model's performance. This analysis allows us to assess the model's sensitivity to changes in epidemic parameters and understand its performance across different scenarios.

All MDP-based models are solved using Gurobi. The RTDP and DP algorithms are implemented in Python. All experiments are conducted on machines equipped with 2.7 GHz Quad-Core Intel Core i7 processors.

\subsection{Model Comparison}
\label{model_comparison}

In this subsection, we present a comprehensive performance comparison between our proposed DRMDP model and traditional MDP-based models for addressing the epidemic control problem.

We refer to the decision-dependent DRMDP model formulated with the McCormick MIP approach in \eqref{drmdp_bellman_mip_set_1} as \texttt{DRMDP-McCormick}, and we denote the decision-dependent DRMDP model formulated with the unary MIP approach in \eqref{drmdp_bellman_mip_set_2} as \texttt{DRMDP-Unary}. In all experiments, we set $k=1000$ for \eqref{drmdp_bellman_mip_set_1} and \eqref{drmdp_bellman_mip_set_2}. For comparison purposes, we consider a classic MDP model with known transition probabilities $\bm{\tilde{p}}_{a\xi}^0$ as in equation \eqref{tilde_pas}. Thus, we solve the Bellman equation $V^t(\xi) = \max_{a \in \mathcal{A}_\xi} \{ \tilde{r}_{a\xi} + \lambda (\bm{\tilde{p}}_{a\xi}^{0})^{\mbox{\tiny T}} \bm{V}^{t+1} \}$ at stage $t$ for this baseline, denoted as \texttt{MDP}. Additionally, we consider a robust MDP model where we use the worst-case transition probabilities $\bm{p}_{a\xi}'$. These probabilities represent the highest probability of transitioning from exposed to infectious while staying within a distance of $0.5$ from the nominal probability $\bm{\tilde{p}}_{a\xi}^0$ in \eqref{tilde_pas}. Consequently, we solve the Bellman equation $V^t(\xi) = \max_{a \in \mathcal{A}_\xi} \{ \tilde{r}_{a\xi} + \lambda (\bm{p}_{a\xi}')^{\mbox{\tiny T}} \bm{V}^{t+1} \}$ at stage $t$ for this robust MDP baseline, which we denote as \texttt{Robust MDP}.

We compute optimal strategies for all models using the RTDP algorithm. For the heuristic function of RTDP, we utilize $h((\xi,T)) = 0, \forall \xi \in \mathcal{\tilde{S}}$ and $h((\xi,t)) = \max_{a \in \mathcal{A}_\xi} \tilde{r}_{a\xi}, \forall \xi \in \mathcal{\tilde{S}}, t = 1,\dots, T-1$. In other words, $h((\xi,t))$ represents only the initial reward, which is less than or equal to the true reward (initial reward plus some non-negative future reward). Therefore, $h((\xi,t))$ is considered \textit{admissible}, and Theorem \ref{rtdp_convergence} holds.

\subsubsection{McCormick and Unary for DRMDP}
\label{mccormick_vs_unary_model}

To determine the optimal formulation for the DRMDP, we compare the computational efficiency and performance of \texttt{DRMDP-McCormick} and \texttt{DRMDP-Unary}.

We conduct experiments with the epidemic control problem, setting $M = 5$, $L = 5$, a time horizon of $T = 12$ (representing a 12-month horizon), a discretization level of $Y = 30$, an initial proportion of exposed individuals $p_E(1) = 0.1$, an initial proportion of susceptible individuals $p_S(1) = 0.6$, and an initial proportion of infectious individuals $p_I(1) = 0.3$. When simulating the spread of the epidemic, we consider two scenarios for the transition probabilities:

\begin{enumerate}
    \item $\bm{\tilde{p}}_{a\xi}^0$ in \eqref{tilde_pas}. In this scenario, the model possesses precise knowledge of the transition probabilities.
    \item  $\bm{q}_{a\xi}$, which satisfies $||\bm{q}_{a\xi} - \bm{\tilde{p}}_{a\xi}^0||_1 \le 0.5$. In this scenario, the transition probabilities are incorrectly specified, and the model lacks knowledge of the true transition probabilities.
\end{enumerate}

The average computational time per iteration for \texttt{DRMDP-Unary} is $149.9s$, while for \texttt{DRMDP-McCormick}, it is $75.1s$. However, the performance of \texttt{DRMDP-Unary} and \texttt{DRMDP}-\texttt{McCormick} is similar. Specifically, for transition probabilities $\bm{\tilde{p}}_{a\xi}^0$, the performance values are $-3.77 \times 10^7$ and $-3.79 \times 10^7$, respectively. For transition probabilities $\bm{q}_{a\xi}$, the performance values are $-3.69 \times 10^7$ and $-3.72 \times 10^7$, respectively. Since \texttt{DRMDP-McCormick} achieves similar performance with only half the runtime, we adopt \texttt{DRMDP-McCormick} for the subsequent experiments, denoting it as \texttt{DRMDP} for brevity.

\subsubsection{DRMDP, MDP and Robust MDP}
\label{mdp_based_model_comparison}

We proceed by comparing the performance of our proposed \texttt{DRMDP} model to classic \texttt{MDP} and \texttt{Robust MDP} models.

We conduct experiments with the epidemic control problem, setting $M = 5$, $L = 5$, a time horizon of $T = 12$ (representing a 12-month horizon), and a discretization level of $Y = 100$. We fix the initial proportion of exposed individuals as $p_E(1) = 0.1$ and consider different initial proportions of susceptible and infectious individuals: $p_S(1) = 0.60, 0.65, 0.70, 0.75$ and $p_I(1) = 0.30, 0.25, 0.20, 0.15$, respectively. Similarly, we consider two different sets of transition probabilities, $\bm{\tilde{p}}_{a\xi}^0$ and $\bm{q}_{a\xi}$, for simulating the epidemic.

The results comparing \texttt{DRMDP}, \texttt{MDP}, and \texttt{Robust MDP} are presented in Figure \ref{accumulated_reward_dd1_versus_mdp}. As shown in Figure \ref{fig:1a}, when the transition probability used in simulation is $\bm{\tilde{p}}_{a\xi}^0$, the total discounted rewards of \texttt{DRMDP}, \texttt{MDP}, and \texttt{Robust MDP} are similar. However, when the transition probability used in simulation is $\bm{q}_{a\xi}$ (Figure \ref{fig:1b}), the total discounted rewards of \texttt{DRMDP} are higher compared to \texttt{MDP} and \texttt{Robust MDP}. These findings demonstrate that \texttt{DRMDP} outperforms \texttt{MDP} and \texttt{Robust MDP} in adapting to environments when the true transition probabilities are unknown.

\begin{figure}[H]
\centering
  \begin{subfigure}{0.45\textwidth}
    \includegraphics[width=0.9\linewidth]{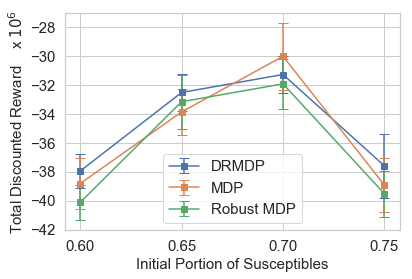}
    \caption{} \label{fig:1a}
  \end{subfigure}%
  \hspace{1cm} 
  \begin{subfigure}{0.45\textwidth}
    \includegraphics[width=0.9\linewidth]{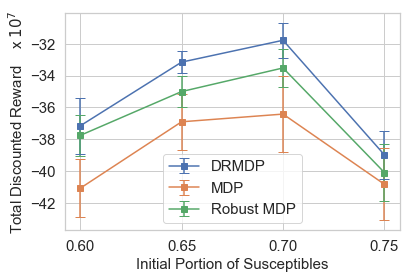}
    \caption{} \label{fig:1b}
  \end{subfigure}%
\caption{Total rewards versus initial proportion of susceptible individuals, averaged across $10$ runs. The error bars show mean $\pm$ standard deviation. (a) the transition probability used in simulation is $\bm{\tilde{p}}_{a\xi}^0$. (b)  the transition probability used in simulation is $\bm{q}_{a\xi}$. }
\label{accumulated_reward_dd1_versus_mdp}
\end{figure}

The optimal actions taken by \texttt{DRMDP}, \texttt{MDP}, and \texttt{Robust MDP} at each stage are summarized in Table \ref{action_table_dd1_versus_mdp}, where the notation $[x,y]$ represents the action of choosing a vaccination level of $x$ and a transmission-reduction level of $y$. Analyzing Table \ref{action_table_dd1_versus_mdp}, we observe that \texttt{DRMDP} administers vaccinations to a greater number of susceptible individuals compared to \texttt{MDP}. Additionally, as the time period progresses, the need for implementing interventions decreases across all models.

Furthermore, with increasing initial proportion of susceptible individuals $p_S(1)$, we note that vaccinations are required in more stages. For instance, in the case of \texttt{DRMDP}, when $p_S(1) = 0.60$, vaccination actions are taken during stages $t = 1$ to $2$, while for $p_S(1) = 0.75$, vaccinations are required during stages $t = 1$ to $3$. This observation indicates a higher demand for vaccinations when a larger proportion of the population is susceptible.

Moreover, the results demonstrate that \texttt{DRMDP} adopts a more aggressive vaccination strategy compared to \texttt{MDP}, implying that the distributionally robust approach favors allocating more resources towards vaccines to ensure a policy that is robust against uncertainty. On the other hand, \texttt{Robust MDP} implements an even stricter policy than \texttt{DRMDP}, resulting in lower total rewards due to the increased cost of intervention. This finding suggests that \texttt{DRMDP}, which considers the worst-case distribution $\mu_{a\xi}$, exhibits a less conservative behavior than \texttt{Robust MDP}.

\begin{table}[H]
\scriptsize 
\renewcommand{\arraystretch}{1}
\caption{Action taken at each stage (averaged across $10$ runs) when the transition probability used in simulation is $\bm{q}_{a\xi}$.}
\begin{adjustbox}{width=\columnwidth,center}
\centering
\begin{tabular}{llllllllllll}
 \hline
  Stage & $t=1$ & $t=2$ & $t=3$ & $t=4$ & $t=5$ & $t=6$ & $t=7$ & $t=8$ & $t=9$ & $t=10$ & $t=11$ \\
 \hline
   DRMDP, $p_S(1) = 0.60$ & [1,5] & [5,5] & [0,0] & [0,0] & [0,0] & [0,0] & [0,0] & [0,0] & [0,0] & [0,0] & [0,0] \\
 \hline
   MDP, $p_S(1) = 0.60$ & [1,5] & [4.6,5] & [0.5,0.5] & [0,0] & [0,0] & [0,0] & [0,0] & [0,0] & [0,0] & [0,0] & [0,0]  \\
 \hline
   Robust MDP, $p_S(1) = 0.60$ & [1,5] & [5,5] & [0,0] & [0,0] & [0,0] & [0,0] & [0,0] & [0,0] & [0,0] & [0,0] & [0,0] \\
 \hline
   DRMDP, $p_S(1) = 0.65$ & [1,5] & [3.8,5] & [3,0] & [0,0] & [0,0] & [0,0] & [0,0] & [0,0] & [0,0] & [0,0] & [0,0] \\
 \hline
   MDP, $p_S(1) = 0.65$ & [1,5] & [3.6,5] & [2.5,2.5] & [0,0] & [0,0] & [0,0] & [0,0] & [0,0] & [0,0] & [0,0] & [0,0]  \\
\hline
   Robust MDP, $p_S(1) = 0.65$ & [1,5] & [4.1,5] & [3,1] & [0,0] & [0,0] & [0,0] & [0,0] & [0,0] & [0,0] & [0,0] & [0,0] \\
 \hline
   DRMDP, $p_S(1) = 0.70$ & [1,5] & [3,5] & [5,0] & [0,0] & [0,0] & [0,0] & [0,0] & [0,0] & [0,0] & [0,0] & [0,0] \\
 \hline
   MDP, $p_S(1) = 0.70$ & [1,5] & [2.7,5] & [5,5] & [0,0] & [0,0] & [0,0] & [0,0] & [0,0] & [0,0] & [0,0] & [0,0]  \\
\hline
   Robust MDP, $p_S(1) = 0.70$ & [1,5] & [3.2,5] & [5,1] & [0,0] & [0,0] & [0,0] & [0,0] & [0,0] & [0,0] & [0,0] & [0,0] \\
 \hline
   DRMDP, $p_S(1) = 0.75$ & [1,5] & [2,5] & [5,0.5] & [0,0] & [0,0] & [0,0] & [0,0] & [0,0] & [0,0] & [0,0] & [0,0] \\
 \hline
   MDP, $p_S(1) = 0.75$  & [1,5] & [2,5] & [4,4] & [1,1] & [0,0] & [0,0] & [0,0] & [0,0] & [0,0] & [0,0] & [0,0] \\ 
 \hline
   Robust MDP, $p_S(1) = 0.75$ & [1,5] & [2.4,5] & [5,1.5] & [0.7,0] & [0,0] & [0,0] & [0,0] & [0,0] & [0,0] & [0,0] & [0,0] \\
 \hline
\end{tabular}
\label{action_table_dd1_versus_mdp}
\end{adjustbox}
\end{table}

The percentage of infectives, recovered individuals, and rewards for \texttt{DRMDP}, \texttt{MDP}, and \texttt{Robust MDP} at each stage, considering the true transition probability $\bm{q}_{a\xi}$, are depicted in Figure \ref{infectives_dd1_versus_mdp}, \ref{recovered_dd1_versus_mdp}, and \ref{reward_dd1_versus_mdp}, respectively. From these figures, several observations can be made.

Firstly, \texttt{DRMDP} exhibits higher stage-wise rewards and demonstrates greater effectiveness in controlling the number of infectives compared to \texttt{MDP}. Specifically, in Figure \ref{infectives_dd1_versus_mdp} and \ref{recovered_dd1_versus_mdp}, both models perform similarly in the early stages ($t=1$ to $4$), but \texttt{DRMDP} achieves zero infectives and full recovery faster than \texttt{MDP} in the intermediate to late stages ($t=5$ to $12$). Similarly, Figure \ref{reward_dd1_versus_mdp} showcases that \texttt{DRMDP} achieves a higher reward earlier than \texttt{MDP}.

These results can be attributed to the fact that the optimal policy derived from the \texttt{MDP} model is based on the assumption that the true transition probabilities are $\bm{\tilde{p}}_{a\xi}^0$, while the policy obtained from the \texttt{DRMDP} model is robust to uncertainty in transition probabilities, which allows it to adapt and perform well even without knowing the true probabilities.

Comparing \texttt{DRMDP} with \texttt{Robust MDP}, we observe that \texttt{DRMDP} is slightly less effective in controlling the number of infectives but achieves a higher stage-wise reward overall. These findings align with our earlier conclusions that \texttt{DRMDP} exhibits a less conservative behavior compared to \texttt{Robust MDP}.

\begin{figure}[H]
  \begin{subfigure}{0.25\textwidth}
    \centering
    \includegraphics[width=0.99\linewidth]{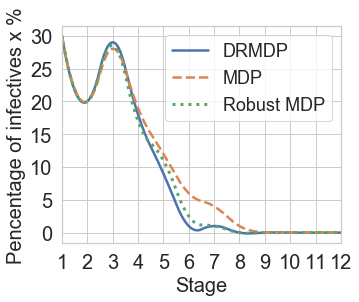}
    \caption{$p_S(1) = 0.6$} \label{fig:4a}
  \end{subfigure}%
  \hfill
  \begin{subfigure}{0.25\textwidth}
    \centering
    \includegraphics[width=0.99\linewidth]{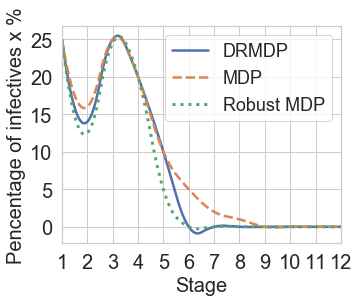}
    \caption{$p_S(1) = 0.65$} \label{fig:4b}
  \end{subfigure}%
  \hfill
    \begin{subfigure}{0.25\textwidth}
    \centering
    \includegraphics[width=0.99\linewidth]{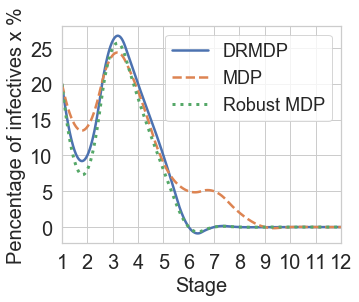}
    \caption{$p_S(1) = 0.7$} \label{fig:4c}
  \end{subfigure}%
\hfill
    \begin{subfigure}{0.25\textwidth}
    \centering
    \includegraphics[width=0.99\linewidth]{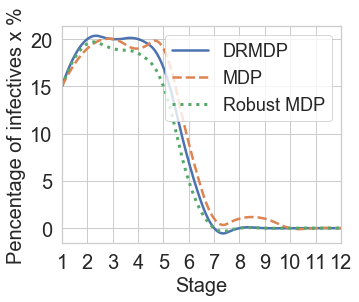}
    \caption{$p_S(1) = 0.75$} \label{fig:4d}
  \end{subfigure}%
\caption{Percentage of infectives versus stage (averaged across $10$ runs) when the transition probability used in simulation is $\bm{q}_{a\xi}$. }
\label{infectives_dd1_versus_mdp}
\end{figure}

\begin{figure}[H]
  \begin{subfigure}{0.25\textwidth}
    \centering
    \includegraphics[width=0.99\linewidth]{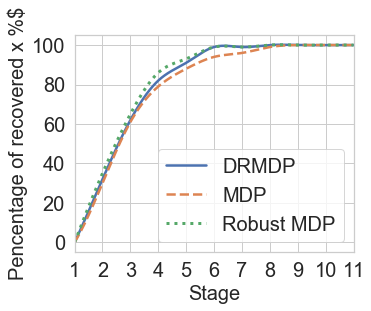}
    \caption{$p_S(1) = 0.6$} \label{fig:5a}
  \end{subfigure}%
  \hfill
  \begin{subfigure}{0.25\textwidth}
    \centering
    \includegraphics[width=0.99\linewidth]{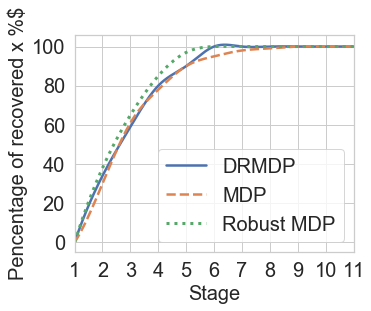}
    \caption{$p_S(1) = 0.65$} \label{fig:5b}
  \end{subfigure}%
  \hfill
    \begin{subfigure}{0.25\textwidth}
    \centering
    \includegraphics[width=0.99\linewidth]{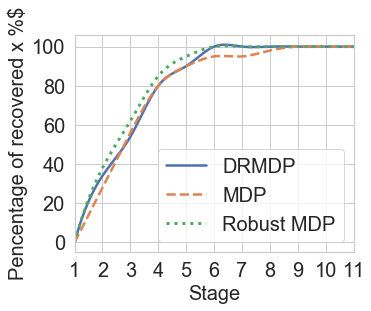}
    \caption{$p_S(1) = 0.7$} \label{fig:5c}
  \end{subfigure}%
\hfill
    \begin{subfigure}{0.25\textwidth}
    \centering
    \includegraphics[width=0.99\linewidth]{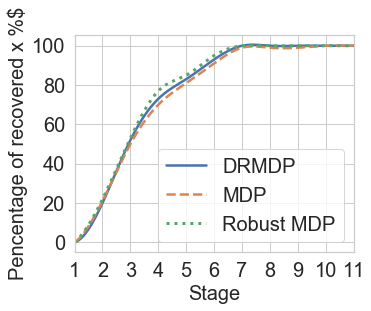}
    \caption{$p_S(1) = 0.75$} \label{fig:5d}
  \end{subfigure}%
\caption{Percentage of recovered versus stage (averaged across $10$ runs) when the transition probability used in simulation is $\bm{q}_{a\xi}$. }
\label{recovered_dd1_versus_mdp}
\end{figure}

\begin{figure}[H]
  \begin{subfigure}{0.25\textwidth}
    \centering
    \includegraphics[width=0.99\linewidth]{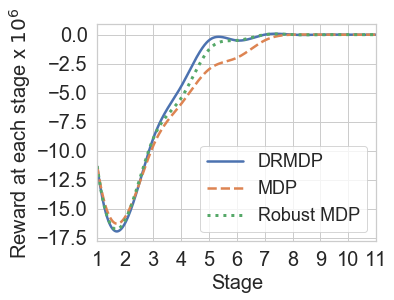}
    \caption{$p_S(1) = 0.6$} \label{fig:6a}
  \end{subfigure}%
  \hfill
  \begin{subfigure}{0.25\textwidth}
    \centering
    \includegraphics[width=0.99\linewidth]{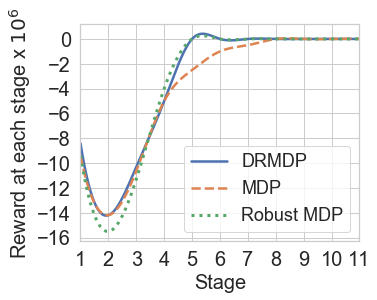}
    \caption{$p_S(1) = 0.65$} \label{fig:6b}
  \end{subfigure}%
  \hfill
  \begin{subfigure}{0.25\textwidth}
    \centering
    \includegraphics[width=0.99\linewidth]{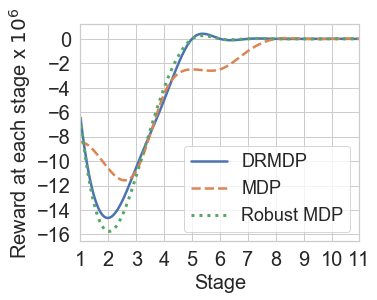}
    \caption{$p_S(1) = 0.7$} \label{fig:6c}
  \end{subfigure}%
    \hfill
  \begin{subfigure}{0.25\textwidth}
    \centering
    \includegraphics[width=0.99\linewidth]{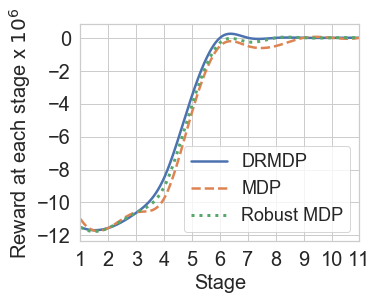}
    \caption{$p_S(1) = 0.75$} \label{fig:6d}
  \end{subfigure}
\caption{Stage-wise reward (averaged across $10$ runs) when the transition probability used in simulation is $\bm{q}_{a\xi}$. }
\label{reward_dd1_versus_mdp}
\end{figure}

\subsection{Algorithm Comparison}
\label{algo_comparison}

To comprehensively evaluate the effectiveness and efficiency of the RTDP algorithm when applied to the \texttt{DRMDP} model, we compare its performance with the conventional Dynamic Programming (DP) approach using backward induction. The DP method for solving the \texttt{DRMDP} model is described in Algorithm \ref{backward_dp}.

\small
\begin{algorithm}[H]
\SetAlgoLined
Initialize $V((\xi,T)) = \tilde{q}_R(\xi)$ for $\xi \in \tilde{\mathcal{S}}$ \\
 \For{$t = T-1, \dots, 1$}{
    \For{$\xi \in \tilde{\mathcal{S}} \cap \mathcal{S}$}{
        Solve MIP relaxation in \eqref{drmdp_bellman_mip_set_1} or \eqref{drmdp_bellman_mip_set_2} to obtain optimal $V^*((\xi,t))$ and $a^*$ \\
        Update $V((\xi,t)) \xleftarrow{} V^*((\xi,t))$ \\
    }
    Update $V((\xi,t)) \xleftarrow{} 0$ for $\xi \in \tilde{\mathcal{S}} \backslash\mathcal{S}$ \\  
 }
 \caption{Conventional dynamic programming}
 \label{backward_dp}
\end{algorithm}
\normalsize

In this comparison, we consider different settings of the state space discretization level, denoted as $Y$, as defined in Section \ref{sec:discretization}. This allows us to assess the impact of state space granularity on the performance of both algorithms. For the RTDP algorithm, we execute a total of 20 iterations when $Y=5$, while for $Y = 10$, $15$, and $20$, we perform 50 iterations. 

As shown in Figure \ref{fig:dp_rewards}, DP yields slightly higher total rewards than RTDP for small discretization levels ($Y=5, 10$). However, as the discretization level increases ($Y=15, 20$), the total rewards of DP and RTDP become similar. Conversely, as depicted in Figure \ref{fig:dp_runtime}, DP requires significantly more computational time than RTDP, especially for high discretization levels. This is due to the fact that DP needs to solve Bellman equations for all states, and the number of states grows polynomially with the discretization level.

\begin{figure}[H]
\centering
  \begin{subfigure}{0.4\textwidth}
    \includegraphics[width=0.85\linewidth]{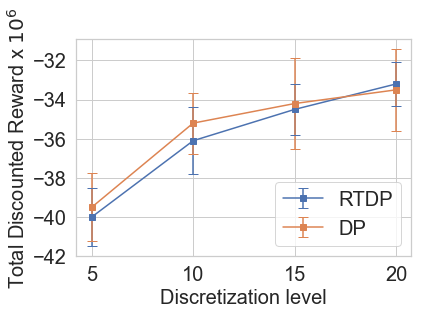}
    \caption{Total rewards} \label{fig:dp_rewards}
  \end{subfigure}%
  \hspace{0.5cm}
  \begin{subfigure}{0.43\textwidth}
    \vspace{0.15cm}
    \includegraphics[width=0.85\linewidth]{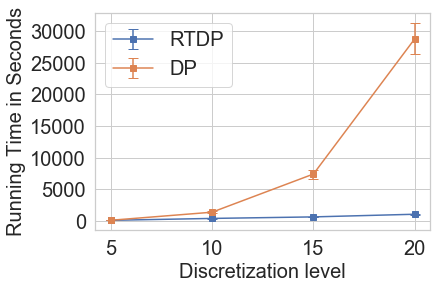}
    \caption{Computational time} \label{fig:dp_runtime}
  \end{subfigure}%
\caption{Performance and runtime comparisons of RTDP and DP. Averaged across $3$ runs. }
\label{fig:rtdp_dp_comparison}
\end{figure}

\subsection{Sensitivity Analysis}
\label{sensitivity}

In this subsection, we analyze the impact of several input parameters on the performance of the \texttt{DRMDP} model for the epidemic control problem. Specifically, we investigate the influence of the following parameters:

\begin{enumerate}
    \item $Q$: the vaccine price.
    \item $k_R$: the cost of increasing the transmission reduction level by one level.
    \item $\mu \beta$: Here, $\mu$ represents the contact rate without any transmission reduction method, and $\beta$ denotes the probability of a susceptible individual becoming infected upon contact with an infectious individual. Since both $\mu$ and $\beta$ collectively affect the number of susceptible individuals who become exposed, we evaluate their product as a combined variable in the sensitivity analysis.
    \item $W$: the health loss plus treatment cost associated with a single infection.
    \item $\alpha_0$: the maximum possible fractional reduction in the contact rate achievable through transmission reduction interventions.
\end{enumerate}

\begin{figure}[H]
  \begin{subfigure}{0.3\textwidth}
    \centering
    \includegraphics[width=0.9\linewidth]{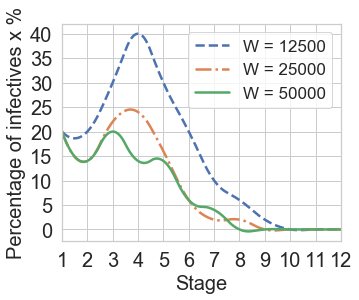}
    \caption{} \label{fig:7a}
  \end{subfigure}%
  \hfill
  \begin{subfigure}{0.3\textwidth}
    \centering
    \includegraphics[width=0.9\linewidth]{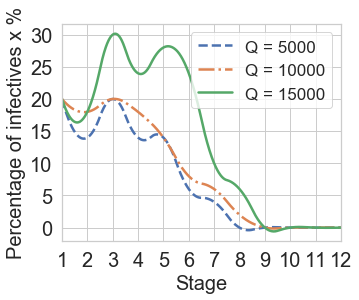}
    \caption{} \label{fig:7b}
  \end{subfigure}%
  \hfill
    \begin{subfigure}{0.3\textwidth}
    \centering
    \includegraphics[width=0.9\linewidth]{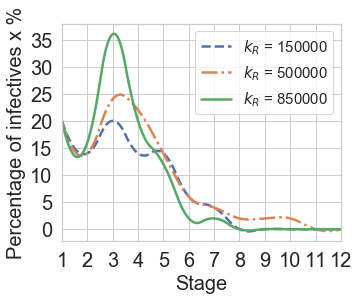}
    \caption{} \label{fig:7c}
  \end{subfigure}%
  
  \begin{subfigure}{0.3\textwidth}
    \centering
    \includegraphics[width=0.9\linewidth]{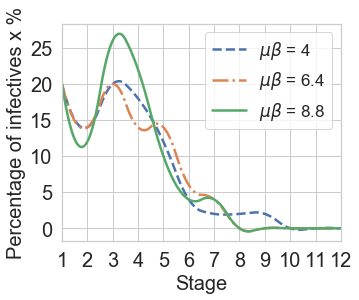}
    \caption{} \label{fig:7d}
  \end{subfigure}%
  \hspace{0.8cm}
  \begin{subfigure}{0.3\textwidth}
    \centering
    \includegraphics[width=0.9\linewidth]{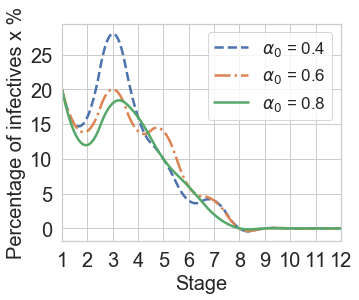}
    \caption{} \label{fig:7e}
  \end{subfigure}%
  
\caption{Sensitivity analysis of \texttt{DRMDP} (solved with RTDP). Percentage of infectives versus stage (averaged across $5$ runs) when the transition probability used in simulation is $\bm{q}_{a\xi}$. }
\label{infective_sensitivity_analysis}
\end{figure}

We kept the value of $k$ fixed at $1000$ in equation \eqref{drmdp_bellman_mip_set_1} and utilized the RTDP algorithm to calculate optimal strategies for the \texttt{DRMDP} model in all our sensitivity analysis experiments. The analysis was performed on the epidemic control problem, considering a time horizon of $T=12$, an initial proportion of exposed individuals $p_E(1)=0.1$, an initial proportion of susceptible individuals $p_S(1)=0.7$, and an initial proportion of infectious individuals $p_I(1)=0.2$.

Figure \ref{infective_sensitivity_analysis} illustrates the percentage of infectives at each stage for different input values. From the figure, we observe that the overall percentage of infectives increases as $Q$, $k_R$, or $\mu\beta$ increases, while it decreases as $W$ or $\alpha_0$ increases.

When the price of vaccines ($Q$) or transmission reduction interventions ($k_R$) increases, the willingness to allocate resources for vaccinations or interventions decreases. Consequently, we expect to see a higher proportion of infected individuals. Additionally, higher values of $\mu$ or $\beta$ lead to easier disease spread, resulting in more infections.

Conversely, an increased value of $W$ indicates a higher penalty for infections, prompting policymakers to prioritize vaccination efforts to reduce the number of infections. Finally, a larger value of $\alpha_0$ signifies the greater effectiveness of transmission reduction methods, leading to a reduced number of infected individuals.

\section{Conclusion}
In this study, we present a Distributionally Robust Markov Decision Process (DRMDP) framework for addressing the dynamic epidemic control problem. Our aim is to provide effective strategies to public health decision-makers that can reduce the proportion of infected individuals while considering cost-efficiency. The proposed DRMDP model is designed to handle the uncertainties inherent in scenarios where the true distribution of disease transmission parameters remains unknown.

To capture the impact of policy actions on the system dynamics, we incorporate an endogenous ambiguity set within our model. This inclusion ensures that the influence of policy-makers' decisions on the system dynamics is comprehensively considered. To solve the problem, we discretize the model and reformulate it as a mixed integer programming (MIP) using either McCormick or unary envelopes. Our computational experiments demonstrate that both approaches yield similar results.

Our findings indicate that the DRMDP outperforms both the classic MDP and robust MDP when evaluated against various metrics, including cost minimization and reduction of infected individuals. This superiority is particularly evident when dealing with scenarios where the true distribution of transition probabilities is unknown. In contrast, the performance of the MDP model deteriorates when the assumption of known transition probabilities is violated. The strength of the DRMDP model lies in its ability to handle uncertain transition probabilities without relying on this assumption, thereby delivering reliable results even in the absence of true probabilities. 

Moreover, we highlight the efficiency and effectiveness of the RTDP algorithm in computing optimal policies based on the DRMDP model. The RTDP algorithm achieves similar performance to traditional DP methods while requiring significantly less computational time.

To conclude our analysis, we conduct a sensitivity analysis that reveals the impact of increasing vaccine and transmission-reduction intervention prices, as well as contact and infection rates, on the proportion of infected individuals. Conversely, we observe that increasing health loss, treatment cost, and intervention effectiveness have the opposite effect.

As a direction for future research, we suggest exploring a continuous action space instead of the current discrete action space. Additionally, while this paper focuses on uncertainty in transition probabilities, future studies could expand their scope to encompass uncertainty in other model parameters.

\bibliographystyle{apalike}
\spacingset{1}
\bibliography{main}

\begin{thebibliography}{}

\bibitem[Barto et~al., 1995]{barto_1995_rtdp}
Barto, A.~G., Bradtke, S.~J., and Singh, S.~P. (1995).
\newblock {Learning to act using real-time dynamic programming}.
\newblock {\em Artificial Intelligence}, 72(1):81 -- 138.

\bibitem[Basciftci et~al., 2021]{basciftci_2020_endogenuous}
Basciftci, B., Ahmed, S., and Shen, S. (2021).
\newblock {Distributionally robust facility location problem under
  decision-dependent stochastic demand}.
\newblock {\em European Journal of Operational Research}, 292(2):548--561.

\bibitem[Bhardwaj et~al., 2020]{bhardwaj_2020_rmdp_covid}
Bhardwaj, A., Ou, H.~C., Chen, H., Jabbari, S., Tambe, M., Panicker, R., and
  Raval, A. (2020).
\newblock {Robust lock-down optimization for COVID-19 policy guidance}.
\newblock In {\em AAAI Fall Symposium}.

\bibitem[Chen et~al., 2019]{chen_2019_drmdp_general}
Chen, Z., Yu, P., and Haskell, W.~B. (2019).
\newblock {Distributionally robust optimization for sequential
  decision-making}.
\newblock {\em Optimization}, 68(12):2397--2426.

\bibitem[Davies, 1996]{scott_1996_kuhn_mdp}
Davies, S. (1996).
\newblock {Multidimensional triangulation and interpolation for reinforcement
  learning}.
\newblock In {\em Advances in Neural Information Processing Systems}, page
  1005–1011.

\bibitem[Gaff and Schaefer, 2009]{gaff_2009_optimalcontrolvariousmodels}
Gaff, H. and Schaefer, E. (2009).
\newblock Optimal control applied to vaccination and treatment strategies for
  various epidemiological models.
\newblock {\em Mathematical Biosciences \& Engineering}, 6(3):469--492.

\bibitem[Giordano et~al., 2020]{giulia_2020_covid}
Giordano, G., Blanchini, F., Bruno, R., Colaneri, P., Filippo, A., Matteo, A.,
  and Colaneri, M. (2020).
\newblock {Modelling the COVID-19 epidemic and implementation of
  population-wide interventions in Italy}.
\newblock {\em Nature Medicine}, 26:855--860.

\bibitem[Grimm et~al., 2021]{grimm_2021_covid_seir}
Grimm, V., Mengel, F., and Schmidt, M. (2021).
\newblock {Extensions of the SEIR model for the analysis of tailored social
  distancing and tracing approaches to cope with COVID-19}.
\newblock {\em Scientific Reports}, 11(4214).

\bibitem[Gupte et~al., 2013]{gupte_2013_unary}
Gupte, A., Ahmed, S., Cheon, M.~S., and Dey, S. (2013).
\newblock {Solving mixed integer bilinear problems using MILP formulations}.
\newblock {\em SIAM Journal on Optimization}, 23(2):721–744.

\bibitem[Holt et~al., 1960]{holt_1960_ldr}
Holt, C.~C., Modigliani, F., Muth, J.~F., and Simon, H.~A. (1960).
\newblock {\em {Planning Production, Inventories, and Work Force}}.
\newblock Prentice Hall.

\bibitem[Hou et~al., 2020]{hou_2020_covid}
Hou, C., Chen, J., Zhou, Y., Hua, L., Yuan, J., He, S., Guo, Y., Zhang, S.,
  Jia, Q., Zhao, C., Zhang, J., Xu, G., and Jia, E. (2020).
\newblock {The effectiveness of quarantine of Wuhan city against the Corona
  Virus Disease 2019 (COVID-19): A well-mixed SEIR model analysis}.
\newblock {\em Journal of Medical Virology}, 92:841--848.

\bibitem[IMF, 2021]{world_economy_outlook}
IMF (2021).
\newblock {World economic outlook Oct 2021}.
\newblock {\em International Monetary Fund}.

\bibitem[Iyengar, 2005]{iyengar_2005_rmdp}
Iyengar, G.~N. (2005).
\newblock {Robust dynamic programming}.
\newblock {\em Mathematics of Operations Research}, 30(2):257--280.

\bibitem[Knobler et~al., 2004]{knobler_2004_sarscost}
Knobler, S., Mahmoud, A., Lemon, S., Mack, A., Sivitz, L., and Oberholtzer, K.
  (2004).
\newblock {\em Learning from SARS: Preparing for the Next Disease Outbreak}.
\newblock The National Academies Press.

\bibitem[Lekone and Finkenstädt, 2006]{lekone_2006_seir}
Lekone, P.~E. and Finkenstädt, B.~F. (2006).
\newblock {Modelling the COVID-19 epidemic and implementation of
  population-wide interventions in Italy}.
\newblock {\em Biometrics}, 62(4):1170--1177.

\bibitem[Luo and Mehrotra, 2020]{luo_2020_endogenuous}
Luo, F. and Mehrotra, S. (2020).
\newblock {Distributionally robust optimization with decision dependent
  ambiguity sets}.
\newblock In {\em Optimization Letters}.

\bibitem[López and Rodó, 2021]{lopez_2021_covid_seir}
López, L. and Rodó, X. (2021).
\newblock {A modified SEIR model to predict the COVID-19 outbreak in Spain and
  Italy: Simulating control scenarios and multi-scale epidemics}.
\newblock {\em Results in Physics}, 21:103746.

\bibitem[Mausam and Kolobov, 2012]{mausam_2012_heuristics}
Mausam and Kolobov, A. (2012).
\newblock {\em {Planning with Markov Decision Processes: An AI Perspective}}.
\newblock Morgan Claypool.

\bibitem[McCormick, 1976]{mccormick_1976_envelope}
McCormick, G.~P. (1976).
\newblock {Computability of global solutions to factorable nonconvex programs:
  Part I — Convex underestimating problems}.
\newblock {\em Mathematical Programming}, 10:147--175.

\bibitem[Moore, 1992]{moore_1992_kuhn}
Moore, D.~W. (1992).
\newblock {Simplicial mesh generation with applications}.
\newblock {\em PhD Thesis, Cornell University, Department of Computer Science}.

\bibitem[Munos and Moore, 2002]{munos_2002_kuhn_mdp}
Munos, R. and Moore, A. (2002).
\newblock {Variable resolution discretization in optimal control}.
\newblock {\em Machine Learning}, 49:291–323.

\bibitem[Nakao et~al., 2021]{nakao_2020_drpomdp}
Nakao, H., Jiang, R., and Shen, S. (2021).
\newblock {Distributionally robust partially observable Markov decision process
  with moment-based ambiguity}.
\newblock {\em SIAM Journal on Optimization}, 31(1):461–488.

\bibitem[Nilim and Ghaoui, 2004]{nilim_2004_rmdp}
Nilim, A. and Ghaoui, L.~E. (2004).
\newblock {Robustness in Markov decision problems with uncertain transition
  matrices}.
\newblock In {\em Advances in Neural Information Processing Systems}, pages
  839--846.

\bibitem[Nilim and Ghaoui, 2005]{nilim_2005_rmdp}
Nilim, A. and Ghaoui, L.~E. (2005).
\newblock {Robust control of Markov decision processes with uncertain
  transition matrice}.
\newblock {\em Operations Research}, 53(5):780--798.

\bibitem[Nouri, 2011]{nouri2011efficient}
Nouri, A. (2011).
\newblock {\em Efficient Model-based Exploration in Continuous State-space
  Environments}.
\newblock Rutgers The State University of New Jersey-New Brunswick.

\bibitem[Noyan et~al., 2022]{noyan_2020_endogenuous}
Noyan, N., Rudolf, G., and Lejeune, M. (2022).
\newblock {Distributionally robust optimization under a decision-dependent
  ambiguity set with applications to machine scheduling and humanitarian
  logistics}.
\newblock {\em INFORMS Journal on Computing}, 34(2):729--751.

\bibitem[Osogami, 2012]{osogami_2012_drmdp_kl}
Osogami, T. (2012).
\newblock {Robustness and risk-sensitivity in Markov decision processes}.
\newblock In {\em Advances in Neural Information Processing Systems}, pages
  233--241.

\bibitem[Parvin et~al., 2012]{parvin_2012_epidemic_control}
Parvin, H., Goel, P., and Gautam, N. (2012).
\newblock {An analytic framework to develop policies for testing, prevention,
  and treatment of two-stage contagious diseases}.
\newblock {\em Annals of Operations Research}, 196(1):707--735.

\bibitem[Pearl, 1984]{pearl_1984_heuristics}
Pearl, J. (1984).
\newblock {\em {Heuristics: Intelligent Search Strategies for Computer Problem
  Solving}}.
\newblock Addison-Wesley Longman Publishing Co., Inc.

\bibitem[Peng et~al., 2020]{peng_2020_covid}
Peng, L., Yang, W., Zhang, D., Zhuge, C., and Hong, L. (2020).
\newblock {Epidemic analysis of COVID-19 in China by dynamical modeling}.
\newblock {\em Arxiv Preprint ArXiv:2002.06563}.

\bibitem[Royset and Wets, 2017]{royset_2017_endogenuous}
Royset, J.~O. and Wets, R.~J. (2017).
\newblock {Variational theory for optimization under stochastic ambiguity}.
\newblock {\em SIAM Journal on Optimization}, 27(2):1118--1149.

\bibitem[Sabbadin and Viet, 2013]{sabbadin_2013_animal_epidemic_control}
Sabbadin, R. and Viet, A.-F. (2013).
\newblock {A tractable leader-follower MDP model for animal disease
  management}.
\newblock In {\em Proceedings of the Twenty-Seventh AAAI Conference on
  Artificial Intelligence}, page 1320–1326.

\bibitem[Sjödin et~al., 2020]{henrik_2020_covid}
Sjödin, H., Wilder-Smith, A., Osman, S., Farooq, Z., and Rocklöv, J. (2020).
\newblock {Only strict quarantine measures can curb the coronavirus disease
  (COVID-19) outbreak in Italy, 2020}.
\newblock {\em Eurosurveillance}, 25(13).

\bibitem[Tang et~al., 2020]{biao_2020_covid}
Tang, B., Xia, F., Tang, S., Bragazzi, N., Li, Q., Sun, X., Liang, J., Xiao,
  Y., and Wu, J. (2020).
\newblock {The effectiveness of quarantine and isolation determine the trend of
  the COVID-19 epidemics in the final phase of the current outbreak in China}.
\newblock {\em International Journal of Infectious Diseases}, 95:288--293.

\bibitem[Viet et~al., 2018]{viet_2018_animal_epidemic_control}
Viet, A.-F., Krebs, S., Rat-Aspert, O., Jeanpierre, L., Belloc, C., and Ezanno,
  P. (2018).
\newblock {A modelling framework based on MDP to coordinate farmers' disease
  control decisions at a regional scale}.
\newblock {\em PLOS ONE}, 13(6):1--20.

\bibitem[White and Eldeib, 1992]{white_1992_rmdp}
White, C.~C. and Eldeib, H.~K. (1992).
\newblock {Markov decision processes with imprecise transition probabilities}.
\newblock {\em Operations Research}, 42(4):739--749.

\bibitem[Xu and Mannor, 2010]{xu_2010_drmdp}
Xu, H. and Mannor, S. (2010).
\newblock {Distributionally robust Markov decision processes}.
\newblock In {\em Advances in Neural Information Processing Systems}, pages
  2505--2513.

\bibitem[Yaesoubi and Cohen, 2011]{yaesoubi_2011_epidemic_control}
Yaesoubi, R. and Cohen, T. (2011).
\newblock {Dynamic health policies for controlling the spread of emerging
  infections: Influenza as an example}.
\newblock {\em PLOS ONE}, 6(9):1--11.

\bibitem[Yan and Zou, 2008]{yan2008_optimalcontrolsars}
Yan, X. and Zou, Y. (2008).
\newblock Optimal and sub-optimal quarantine and isolation control in sars
  epidemics.
\newblock {\em Mathematical and Computer Modelling}, 47(1):235--245.

\bibitem[Yang, 2017a]{yang_2017_drmdp_wass}
Yang, I. (2017a).
\newblock {A convex optimization approach to distributionally robust Markov
  decision processes with Wasserstein distance}.
\newblock {\em IEEE Control Systems Letters}, 1(1):164--169.

\bibitem[Yang, 2017b]{yang_2017_drmdp_moment}
Yang, I. (2017b).
\newblock {A dynamic game approach to distributionally robust safety
  specifications for stochastic systems}.
\newblock {\em Automatica}, 94:94--101.

\bibitem[Yu and Xu, 2016]{yu_2016_drmdp_counterpart}
Yu, P. and Xu, H. (2016).
\newblock {Distributionally robust counterpart in Markov decision processes}.
\newblock {\em IEEE Transactions on Automatic Control}, 61(9):2538 -- 2543.

\bibitem[Yu and Shen, 2022]{yu_2020_endogenuous}
Yu, X. and Shen, S. (2022).
\newblock {Multistage distributionally robust mixed-integer programming with
  decision-dependent moment-based ambiguity sets}.
\newblock {\em Mathematical Programming}, 196(1–2):1025–1064.

\bibitem[Zhang et~al., 2016]{zhang_2016_endogenuous}
Zhang, J., Xu, H., and Zhang, L. (2016).
\newblock {Quantitative stability analysis for distributionally robust
  optimization with moment constraints}.
\newblock {\em SIAM Journal on Optimization}, 26:1855--1882.

\bibitem[Zhao and Chen, 2020]{zhao_2020_covid}
Zhao, S. and Chen, H. (2020).
\newblock {Modeling the epidemic dynamics and control of COVID-19 outbreak in
  China}.
\newblock {\em Quantative Biology}, pages 1--9.

\end{thebibliography}

\newpage
\appendix

\pagenumbering{gobble}

\section*{Supplemental Online
Materials to ``Decision-Dependent Distributionally Robust Markov Decision Process Method in Dynamic Epidemic Control" by Jun Song, William Yang and Chaoyue Zhao.}

\vspace{0.5cm}

\thmambiguitysetone*

\begin{proof}
We first obtain the dual formulation of \eqref{inner_problem_set_1} as follows:
\begin{subequations}
    \begin{align}
        \max_{z, \bm{w}, \bm{u}}L'(z,\bm{w},\bm{u}) = \hspace{2mm}& \tilde{r}_{a\xi} + z - \bm{w}^{\mbox{\tiny T}} \bm{\tilde{\eta}}^U_{a\xi} + \bm{u}^{\mbox{\tiny T}} \bm{\tilde{\eta}}^L_{a\xi} \label{dual_obj}\\
        s.t. \hspace{5mm} & \lambda \bm{p}_{a\xi}^{\mbox{\tiny T}} \bm{V}^{t+1} - z + \bm{w}^{\mbox{\tiny T}} \bm{p}_{a\xi} - \bm{u}^{\mbox{\tiny T}} \bm{p}_{a\xi} \ge 0, \hspace{5mm} \forall \bm{p}_{a\xi} \in \Delta(\mathcal{\tilde{S}}),\label{inner_problem_dual_set_1_subeq1}\\
        & \bm{w} + \bm{u} \le k\bm{1}, \\
        & \bm{w}, \bm{u} \ge \bm{0}
    \end{align}
    \label{dual_appendix}
\end{subequations}
where $z$ with unrestricted sign, $\bm{w} \geq 0$, and $\bm{u}\geq 0$ are dual variables for constraints \eqref{inner_prob_prob}, \eqref{inner_prob_u}, and \eqref{inner_prob_l} respectively. Next, we prove that the strong duality is met. If there exists a point $\bar{\bm{x}},\bar{\mu}_{a\xi}$ in the feasible region of \eqref{inner_problem_set_1} such that all inequality constraints  are non-binding, the Slater condition is satisfied and strong duality holds. Note that for any  $\mu_{a\xi},\bm{x}$ that are in the feasible region of \eqref{inner_problem_set_1}, if one of the inequality constraints \eqref{inner_prob_u} or \eqref{inner_prob_l} is binding, we can let $\bar{\bm{x}} = \bm{x}+\epsilon$ for any $\epsilon >0$. Letting $\bar{\mu}_{a\xi} = \mu_{a\xi}$ we have $\bar{\bm{x}},\bar{\mu}_{a\xi}$ satisfy the Slater condition. Therefore strong duality holds and the primal and dual objectives \eqref{inner_prob_obj} and \eqref{dual_obj} are equal.

For some $z,\bm{w}$ and $\bm{u}$, \eqref{inner_problem_dual_set_1_subeq1} is satisfied if its left-hand-side is non-negative for all $\bm{p}_{a\xi} \in \Delta(\mathcal{\tilde{S}})$. Therefore, this is equaivalent to the left hand side being non-negative for the minimum such $\bm{p}_{a\xi}$ that is in the probability simplex. Therefore \eqref{inner_problem_dual_set_1_subeq1} can be reformulated as the following. 
\begin{equation}
\begin{split}
    \min_{\bm{p}_{a\xi} \ge 0} \hspace{5mm} &\lambda \bm{p}_{a\xi}^{\mbox{\tiny T}} \bm{V}^{t+1} - z + \bm{w}^{\mbox{\tiny T}} \bm{p}_{a\xi} - \bm{u}^{\mbox{\tiny T}} \bm{p}_{a\xi} \ge 0 \\
    s.t. \hspace{5mm} & \bm{1}^{\mbox{\tiny T}} \bm{p}_{a\xi} = 1.\\
    \label{constraint_opt}
\end{split}
\end{equation}
Next, we take the dual of the \eqref{constraint_opt}. We introduce the  dual variable $q$, and using the same dualization procedure described before, we arrive at the following formulation for the dual of \eqref{constraint_opt}:
\begin{equation}
\begin{split}
    \max_{q} \hspace{5mm} & q-z \ge 0 \\
    s.t. \hspace{5mm} &q\bm{1}  \le \lambda \bm{V}^{t+1} + \bm{w} - \bm{u}.
    \label{dual_of_constraint}
\end{split}
\end{equation}
Substituing \eqref{dual_of_constraint} into \eqref{inner_problem_dual_set_1_subeq1}, we arrive at the following reformulation of  \eqref{dual_appendix}:
\begin{equation*}
    \begin{split}
        \max_{z, \bm{w}, \bm{u}, q} \hspace{5mm} & \tilde{r}_{a\xi} + z - \bm{w}^{\mbox{\tiny T}} \bm{\tilde{\eta}}^U_{a\xi} + \bm{u}^{\mbox{\tiny T}} \bm{\tilde{\eta}}^L_{a\xi} \\
        s.t. \hspace{5mm} & q-z \ge 0, \\
        & q\bm{1}  \le \lambda \bm{V}^{t+1} + \bm{w} - \bm{u}, \\
        & \bm{w} + \bm{u} \le k\bm{1}, \\
        & \bm{w}, \bm{u} \ge \bm{0},
    \end{split}
\end{equation*} which can be further simplified as the following because $q=z$ at optimality:
\begin{equation*}
    \begin{split}
        \max_{\bm{w}, \bm{u}, q} \hspace{5mm} & \tilde{r}_{a\xi} + q - \bm{w}^{\mbox{\tiny T}} \bm{\tilde{\eta}}^U_{a\xi} + \bm{u}^{\mbox{\tiny T}} \bm{\tilde{\eta}}^L_{a\xi} \\
        s.t. \hspace{5mm}
        & q\bm{1}  \le \lambda \bm{V}^{t+1} + \bm{w} - \bm{u}, \\
        & \bm{w} + \bm{u} \le k\bm{1}, \\
        & \bm{w}, \bm{u} \ge \bm{0}.
    \end{split}
\end{equation*}

We have now reformulated the inner problem. The final step is to combine this result with  the outer maximization problem that includes the action $a\in \mathcal{A}$ as a decision variable to obtain a reformulation of     \eqref{drmdp_bellman}
\begin{equation*}
    \begin{split}
        V^{t}(\xi) = \max_{a \in \mathcal{A}, \bm{w}, \bm{u}, q} & \tilde{r}_{a\xi} + q - \bm{w}^{\mbox{\tiny T}} \bm{\tilde{\eta}}^U_{a\xi} + \bm{u}^{\mbox{\tiny T}} \bm{\tilde{\eta}}^L_{a\xi} \\
        s.t. \hspace{5mm} 
        & q\bm{1}  \le \lambda \bm{V}^{t+1} + \bm{w} - \bm{u},  \\
        & \bm{w} + \bm{u} \le k\bm{1}, \\
        & \bm{w}, \bm{u} \ge \bm{0}.
    \end{split}
    \end{equation*}
\end{proof}

\end{document}